\documentclass[11pt]{amsart}

\usepackage[sc]{mathpazo}
\usepackage[margin=1in]{geometry}
\usepackage{amssymb}
\usepackage{hyperref}
\usepackage{amsmath,amscd}
\usepackage[all,cmtip]{xy}
\usepackage{qtree}
\usepackage{color}
\usepackage{array}
\newcolumntype{C}{>{$}c<{$}}

\newtheorem{theorem}{Theorem}[section]
\newtheorem{lemma}[theorem]{Lemma}
\newtheorem{proposition}[theorem]{Proposition}
\newtheorem{corollary}[theorem]{Corollary}

\theoremstyle{definition}

\newtheorem{example}[theorem]{Example}

\theoremstyle{remark}
\newtheorem{remark}[theorem]{Remark}

\def\ZZ{{\mathbb Z}}

\def\RR{{\mathbb R}}
\def\CC{{\mathbb C}}
\def\FF{{\mathbb F}}

\def\SS{{\mathfrak S}}

\def\supp{\mathrm{supp}}
\def\Aut{\mathrm{Aut}}
\def\GL{\mathrm{GL}} 
 
\def\tr{\mathrm{tr}}

\DeclareMathOperator{\Mat}{Mat}

\DeclareMathOperator{\rank}{rank}

\def\id{\mathrm{id}}
\def\T{\mathcal{T}}
\def\qand{\quad\text{and}\quad}

\newcommand\qbinom[2]{\left[\begin{matrix} #1 \\ #2 \end{matrix} \right]}

\begin{document}

\title{Norton algebras of the Hamming Graphs via linear characters} 
\author{Jia Huang}
\address{Department of Mathematics and Statistics, University of Nebraska, Kearney, NE 68849, USA}
\email{huangj2@unk.edu}

\pagestyle{plain}

\begin{abstract}
The Norton product is defined on each eigenspace of a distance regular graph by the orthogonal projection of the entry-wise product.
The resulting algebra, known as the Norton algebra, is a commutative nonassociative algebra that is useful in group theory due to its interesting automorphism group.
We provide a formula for the Norton product on each eigenspace of a Hamming graph using linear characters.
We construct a large subgroup of automorphisms of the Norton algebra of a Hamming graph and completely describe the automorphism group in some cases.
We also show that the Norton product on each eigenspace of a Hamming graph is as nonassociative as possible, except for some special cases in which it is either associative or equally as nonassociative as the so-called double minus operation previously studied by the author, Mickey, and Xu.
Our results restrict to the hypercubes and extend to the halved and/or folded cubes, the bilinear forms graphs, and more generally, all Cayley graphs of finite abelian groups.
\end{abstract}



\maketitle

\section{Introduction}

Distance regular graphs have many nice algebraic and combinatorial properties and have been extensively studied. 
For instance, (the adjacency matrix $A$ of ) a distance regular graph $\Gamma=(X,E)$ with vertex set $X$ and edge set $E$ has $d+1$ distinct eigenvalues $\theta_0 > \theta_1 > \cdots > \theta_d$ and the corresponding eigenspaces $V_0, V_1, \ldots, V_d$ form a direct sum decomposition of the vector space $\RR^X:= \{ f: X\to \RR \} \cong \RR^{|X|}$, where $d$ is the diameter of $\Gamma$.
Furthermore, there is a general method to obtain the eigenvalues and eigenspaces of a distance regular graph; see, for example, Brouwer, Cohen and Neumaier~\cite[\S4.1]{DistReg1}.

One can define an interesting product on each eigenspace $V_i$ of a distance regular graph $\Gamma$ by doing the entry-wise product of two eigenvectors in $V_i$ and projecting the resulting vector back to $V_i$.
The gives an algebra, known as the \emph{Norton algebra}, which is commutative but not necessarily associative.
It was studied in group theory due to its interesting automorphism group~\cite{Norton1, Norton2}.

When $\Gamma$ belongs to certain important families of distance regular graphs (i.e., the Johnson graphs, Grassmann graphs, dual polar graphs, and hypercube graphs), Levstein, Maldonado and Penazzi~\cite{DualPolarGraph,NortonAlgebra} constructed the eigenspaces from a filtration of vector spaces corresponding to a graded lattice associated with $\Gamma$, and derived an explicit formula for the Norton product on the eigenspace of $V_1$.
Recently Terwilliger~\cite{Terwilliger} obtained a more general formula for $Q$-polynomial distance-regular graphs. 
But for $i\ge2$ the Norton algebra structure on $V_i$ has not been determined.

In this paper we focus on the \emph{Hamming graph} $H(n,e)$, whose vertex set $X$ consists of all words of length $n$ on the alphabet $\{0,1,\ldots,e-1\}$ and whose edge set $E$ consists of all unordered pairs of vertices differing in exactly one position.
As an important family of distance regular graphs, the Hamming graphs are useful in multiple branches of mathematics and computer science.
Their eigenvalues are well known~\cite[\S9.2]{DistReg1} and their eigenspaces have been investigated from various perspectives.
For example, Valyuzhenich and Vorob'ev~\cite{Supp} studied the minimum cardinality of the support of an eigenvector of a Hamming graph, and for certain Hamming graphs with special parameters, Sander~\cite{Sander} constructed bases for their eigenspaces using vectors over $\{0,1,-1\}$.

If we allow extension of scalars to the complex field $\CC$, there is a nice complex eigenbasis of the Hamming graph $H(n,e)$ consisting of all linear characters of its vertex set $X=\ZZ_e^n$ viewed as a group, and a real eigenbasis can be obtained by taking the real and imaginary parts of these characters.
In fact, this is valid for any Cayley graph of a finite abelian group; see, for example, Lov\'{a}sz~\cite[Exercise~11.8]{Lovasz79} and DeVos--Goddyn--Mohar--\v{S}\'{a}mal~\cite{CayleySum}.
It will be an interesting problem to determine whether an eigenbasis over $\ZZ$ (or even over $\{0,1,-1\}$) can be obtained.

As an application of the linear character approach, we provide a formula for the Norton product on each eigenspace $V_i$ of the Hamming graph $H(n,e)$, and use this formula to study the automorphism group $\Aut(V_i)$ of the Norton algebra $V_i$.
It is known that the automorphisms of the Hamming graph $H(n,e)$ form a group isomorphic to the wreath product $\SS_e \wr \SS_n$~\cite[Theorem~9.2.1]{DistReg1}.
We show that $\Aut(V_1)\cong \SS_e\wr\SS_n$ by constructing all idempotents in $V_1$, but $\Aut(V_i)$ could be much smaller or bigger than this group is $i\ne 1$.
In general, we construct a large subgroup of $\Aut(V_i)$, which is related to the wreath product of the semidirect product $\ZZ_e\rtimes\ZZ_e^\times$ of the group $\ZZ_e$ and its multiplicative group $\ZZ_e^\times$ with the symmetric group $\SS_n$.
A complete description of $\Aut(V_i)$ for $i\ge2$ will be a problem for future study.

We also determine the extent to which the Norton product on $V_i$ is nonassociative. 
For a given binary operation $*$ on a set $Z$, let $C_{*,m}$ be the number of distinct results obtained by inserting parentheses into the expression $z_0* z_1* \cdots * z_m$, where $z_0, z_1,\ldots, z_m$ are $Z$-valued indeterminates.
It is well known that $C_{*,m}$ is bounded above by the ubiquitous \emph{Catalan number} $C_m:=\frac{1}{m+1}\binom{2m}{m}$.
We have $C_{*,m}=1$ for all $m\ge0$ if and only if $*$ is associative, and we say $*$ is \emph{totally nonassociative} if $C_{*,m}=C_m$ for all $m\ge0$.
In general, $C_{*,m}$ measures how far the operation $*$ is from being associative.  
The sequence $(C_{*,m})_{m\ge0}$ was called the \emph{associative spectrum} of the binary operation $*$ by Cs\'{a}k\'{a}ny and Waldhauser~\cite{AssociativeSpectra1}.
Braitt and Silberger~\cite{Subassociative} studied this sequence for a groupoid $(G,*)$ and called it the \emph{subassociativity type} of $(G,*)$.
Independently, Hein and the author~\cite{CatMod} also proposed the study of $C_{*,m}$ for a binary operation $*$.
For further investigations of this nonassociativity measurement, see, e.g., Hein and the author~\cite{VarCat} and Liebscher and Waldhauser~\cite{AssociativeSpectra2}.

We show that the Norton product $\star$ on each eigenspace $V_i$ is totally nonassociative except for some special cases, in which it is either associative for some trivial reasons or equally as nonassociative as the \emph{double minus operation} $\ominus$ defined by $a\ominus b:=-a-b$ for all $a,b\in\CC$, in the sense that any two ways to parenthesize $z_0\ominus z_1\ominus \cdots \ominus z_m$ produce distinct results if and only if so do the same two ways to parenthesize of $z_0\star z_1\star\cdots \star z_m$.
Therefore in the last case we have $C_{\star,m}=C_{\ominus,m}$ given by the sequence A000975~\cite{A000975} in OEIS~\cite{OEIS}, according to Cs\'{a}k\'{a}ny and Waldhauser~\cite{AssociativeSpectra1} and work of the author, Mickey, and Xu~\cite{DoubleMinus}.

Below is a summary of our results on the Norton algebra of the Hamming graph $H(n,e)$.

\begin{theorem}\label{thm:Hamming}
For $i=0,1,\ldots,n$, the (complex) Norton algebra $V_i$ of $H(n,e)$ satisfies the following.
\begin{itemize}
\item
It has a basis $\left\{\chi_u : u\in X_i \right\}$, where $X_i$ is the set of elements in $X=\ZZ_e^n$ with exactly $i$ nonzero entries, such that if $u,v\in X_i$ then
\[ \chi_u \star \chi_v =
\begin{cases}
\chi_{u+v} & \text{if } u+v\in X_i \\
0 & \text{otherwise}.
\end{cases} \]
\item
For $e\ge3$, its automorphism group is trivial if $i=0$, is isomorphic to $\SS_e\wr\SS_n$ if $i=1$ or $\SS_3\wr\SS_{2^{n-1}}$ if $i=n$ and $e=3$, and admits a subgroup isomorphic to $(\ZZ_e\rtimes\ZZ_e^\times) \wr \SS_n$ if $i\ge1$. 
\item
Its product $\star$ is associative if $i=0$, equally as nonassociative as the double minus operation $\ominus$ if $e=3$ and $i\in\{1,n\}$, or totally nonassociative if $e=3$ and $1<i<n$ or if $e\ge4$ and $1\le i\le n$.
\end{itemize} 
\end{theorem}

The nonassociativity measurement in Theorem~\ref{thm:Hamming} is very similar to what we obtained in previous work~\cite{Huang20} on the Norton product on the eigenspace $V_1$ of the Johnson graphs, Grassmann graphs and dual polar graphs based on the formulas by Levstein, Maldonado and Penazzi~\cite{DualPolarGraph,NortonAlgebra}.
The results on the automorphism group and nonassociativity in Theorem~\ref{thm:Hamming} do not include the case $e=2$ as it is somewhat different from the case $e\ge3$.

In fact, the Hamming graph $H(n,2)$ is the well-known \emph{hypercube} $Q_n$, and the linear characters of its vertex set $X=\ZZ_2^n$ are all real (actually over $\{0,1,-1\}$).
Furthermore, as the hypercube $Q_n$ is both bipartite and antipodal, it can be halved and folded~\cite[\S9.2.D]{DistReg1}.
The \emph{halved cube} or \emph{half-cube} $\frac12 Q_n$ can be obtained from the hypercube $Q_n$ by selecting vertices with an even number of ones and drawing edges between pairs of vertices differing in exactly two positions.
The \emph{folded cube} $\square_n$ can be obtained from the hypercube graph $Q_n$ by identifying each pair of vertices at distance $n$ from each other.
Applying both constructions above to the hypercube $Q_n$ gives the \emph{folded half-cube} $\frac12 \square_n$.
These graphs are all distance regular and their eigenvalues are known~\cite[\S9.2.D]{DistReg1}.

Since these graph are also Cayley graphs of finite abelian groups, we can study their Norton algebras using the same method as for the Hamming graphs, with the linear characters naturally indexed by certain sets and the Norton product of two linear characters $\chi_S$ and $\chi_T$ determined by the \emph{symmetric difference} $S\triangle T:=(S-T)\cup(T-S)$ of the indexing sets $S$ and $T$.
The automorphism groups of their Norton algebras are related to the \emph{hyperoctahedral group} $\SS_n^B \cong \ZZ_2\wr\SS_n$ (the Coxeter group of type $B_n$) consisting of all bijections $f$ on the set $\{\pm1,\ldots,\pm n\}$ satisfying $f(-j)=-f(j)$ for all $j\in[n]$, its center $\{\pm1\}$ consisting of the constant functions $f=\pm1$, and its subgroup $\SS_n^D$ (the Coxeter group of type $D_n$) consisting of all $f\in\SS_n^B$ with $|\{j\in[n]: f(j)<0\}|$ even.
Our results are summarized below, where $V_i(\Gamma)$ denotes the $i$th eigenspace and the corresponding Norton algebra of a distance regular graph $\Gamma$.

\begin{theorem}\label{thm:hypercube}
For $i=0,1,\ldots,n$, the (real) Norton algebra $V_i(Q_n)$ satisfies the following.
\begin{itemize}
\item
It has a basis $\{\chi_S: S\subseteq[n],\ |S|=i\}$ such that for all $S,T\subseteq[n]$ with $|S|=|T|=i$,
\[ \chi_S \star \chi_T =
\begin{cases}
\chi_{S\triangle T} & \text{if } |S\triangle T|=i \\
0 & \text{otherwise}.
\end{cases} \]
\item
Its automorphism group is trivial if $i=0$, 
equals the general linear group of the underlying vector space if $i>\lfloor 2n/3 \rfloor$ or $i$ is odd,
and admits $\SS_n^B/\{\pm1\}$ as a subgroup if $1\le i<n$ is even.
\item
Its product is associative if $i=0$, $i>\lfloor 2n/3 \rfloor$ or $i$ is odd, but totally nonassociative otherwise.
\end{itemize}
For $i=0,1,\ldots,\lfloor n/2 \rfloor$, there is an algebra isomorphism $V_i(\square_n)\cong V_{2i}(Q_n)$.
\end{theorem}

\begin{theorem}\label{thm:halved-cube}
For $i=0,1,\ldots,\lfloor n/2 \rfloor$, the (real) Norton algebra $V_i(\frac12 Q_n)$ satisfies the following.
\begin{itemize}
\item
It has a basis $\{\chi_S: S\subseteq[n],\ |S|=i\}$ if $0\le i<n/2$ or $\{\chi_S: S\subseteq[n],\ |S|=i,\ 1\in S\}$ if $i=n/2$ such that for all $S,T\subseteq[n]$ with $|S|=|T|=i$ we have
\[ \chi_S \star \chi_T =
\begin{cases}
\chi_{S\triangle T} & \text{if } |S\triangle T|=i \\
\chi_{(S\triangle T)^c} & \text{if } |S\triangle T|=n-i \\
0 & \text{otherwise.} 
\end{cases} \]
Hence $V_i(\frac12 Q_n)\cong V_i(Q_n)$ if $i< \lceil n/3 \rceil$ or $n-i$ is odd, and $V_2(\frac12 Q_4) \cong V_2(Q_3) \cong V_1(\frac12 Q_3)$.
\item
Its automorphism group admits a subgroup isomorphic to $\SS_n^D$ if $i$ and $n$ are not both even and $(i,n)\ne(1,2)$, or $\SS_n^D/\{\pm1\}$ if $i$ and $n$ are both even and $(i,n)\ne(2,4)$.
\item
Its product $\star$ is associative if $i=0$, $i< \lceil n/3 \rceil$ is odd, $i$ and $n-i$ are both odd, or $n-i$ is odd and $i>\lfloor 2n/3\rfloor$, but totally nonassociative otherwise.
\end{itemize}
For $i=0,1,\ldots,d=\lfloor n/4 \rfloor$, there is an algebra isomorphism $V_i(\frac12 \square_n)\cong V_{2i}(\frac12 Q_n)$.
\end{theorem}

Our method is also valid for the \emph{bilinear forms graph} $H_q(d,e)$, whose vertex set $X=\Mat_{d,e}(\FF_q)$ consists of all $d\times e$ matrices over a finite field $\FF_q$ and whose edge set $E$ consists of unordered pairs $xy$ of vertices $x,y\in X$ with $\rank(x-y)=1$.
This is a distance regular graph of diameter $d$ (assuming $d\le e$) and can be viewed as a $q$-analogue of the Hamming graph $H(d,e)$~\cite[\S9.5.A]{DistReg1}.  
If $d=1$ then $H_q(d,e)$ is a complete graph isomorphic to the Hamming graph $H(1,q^e)$.
In general, $H_q(d,e)$ is a Cayley graph of the finite abelian group $X$, so the linear character approach applies.

\begin{theorem}
For $i=0,1,\ldots,d\le e$, the (complex) Norton algebra $V_i(H_q(d,e))$ satisfies the following.
\begin{itemize}
\item
It has a basis $\{\chi_u: u\in X,\ \rank(u)=i\}$ such that 
\[ \chi_u\star\chi_v =
\begin{cases}
\chi_{u+v} & \text{if } \rank(u+v)=i \\
0 & \text{otherwise}.
\end{cases} \]
\item
Its automorphism group admits subgroups isomorphic to 
\[ \Mat_{d,e}(\FF_q)\rtimes\left(\GL_d(\FF_q) \times \GL_e(\FF_q) / \{(cI_d, cI_e): c\in\FF_q^\times\} \right).\]
\item
Its product is associative if $i=0$ or totally associative if $d\ge2$ and $i=1,\ldots,d$.
\end{itemize}
\end{theorem}

This paper is structured as follows.
In Section~\ref{sec:Cayley} we review the linear character eigenbasis for any Cayley graph of a finite abelian group and use this basis to establish a formula for the Norton product on each eigenspace. 
Then we carefully examine the Norton algebras of the Hamming graphs in Section~\ref{sec:Hamming}, investigate their automorphisms groups in Section~\ref{sec:Aut}, and measure their nonassociativity in Section~\ref{sec:NonAssoc}.
We extend our results to the halved and/or folded cubes in Section~\ref{sec:cubes} and to the bilinear forms graphs in Section~\ref{sec:Hq}.
We conclude this paper with some remarks and questions in Section~\ref{sec:questions}.


\section{Cayley graphs of finite abelian groups}\label{sec:Cayley}

In this section we study the Norton algebras of a Cayley graph of a finite abelian group using the linear characters of the group.
First recall that a \emph{distance regular graph} is a graph with distance $d(-,-)$ such that the number of vertices $z$ with $d(x,z)=i$ and $d(y,z)=j$ depends only on $i$, $j$, and $k=d(x,y)$, but not on the choices of the vertices $x$ and $y$.
Such a graph satisfies many nice properties and has been an important topic in algebraic combinatorics; see, for example, Brouwer--Cohen--A. Neumaier~\cite{DistReg1} and van Dam--Koolen--Tanaka~\cite{DistReg2}.

In particular, if $\Gamma$ is a distance regular graph of diameter $d$, then there are exactly $d+1$ (real) eigenvalues $\theta_0>\theta_1>\cdots>\theta_d$ and the corresponding (real) eigenspaces $V_0, V_1, \ldots, V_d$ form a direct sum decomposition of $\RR^X:=\{f: X\to \RR\} \cong \RR^{|X|}$.
Using the orthogonal projection $\pi_i: \RR^X\to V_i$ one can define the \emph{Norton product} $\star$ on each eigenspace $V_i$ by $\tau\star\tau':= \pi_i(\tau\cdot\tau')$ for all $\tau, \tau'\in V_i$, where $\tau\cdot\tau'$ is the entry-wise product, i.e., $(\tau\star\tau')(x):= \tau(x)\tau'(x)$ for all $x\in X$.
This gives a commutative but not necessarily associative algebra known as the \emph{Norton algebra}, which turns out to be useful in group theory~\cite{Norton1, Norton2}.

For some important families of distance regular graphs (the Johnson graphs, Grassmann graphs, dual polar graphs, and  hypercube graphs), Levstein, Maldonado and Penazzi~\cite{DualPolarGraph,NortonAlgebra} used a graded lattice to construct a filtration of vector spaces and obtained a formula for the Norton product on the eigenspace of the second largest eigenvalue. 
Based on this formula we~\cite{Huang20} studied the nonassociativity of the Norton product, with the result on the hypercube graphs extended to the Hamming graphs.
For the other eigenspaces of these graphs and for the eigenspaces of other distance regular graphs, the Norton algebra structure has not been determined yet.

Now we use linear characters to give a complete description of all Norton algebras of a distance regular graph if it is also a Cayley graph of a finite abelian group, with the ground field $\RR$ extended to the complex field $\CC$.
To this end, we review some basic properties of linear characters. 

Let $G$ be a group and let $R^\times$ denote the multiplicative group of all units in a ring $R$.
A \emph{linear character} of $G$ is a group homomorphism $\chi:G\to\CC^\times$.
The linear characters of $G$ form an abelian group $G^*$ under the \emph{entry-wise product} defined by 
\[ (\chi\cdot\chi')(g) := \chi(g)\chi'(g) \quad \text{for all $\chi,\chi'\in G^*$ and $g\in G$}. \] 
The following result is well known and we include a short proof here for completeness.

\begin{theorem}\label{thm:character}
The group $G^*$ of all linear characters of a finite abelian group $G$ is isomorphic to $G$ and is an orthonormal basis for the space $\CC^G:=\{\phi: G\to \CC\}\cong \CC^{|G|}$ endowed with the inner product
\[ \langle \phi,\psi \rangle := \frac{1}{|G|} \sum_{g\in G} \phi(g) \overline{\psi(g)}
\quad\text{ for all $\phi,\psi\in \CC^G$.} \]
\end{theorem}

\begin{proof}
For any $\chi\in G^*$ and any $g\in G$ with order $n$, we have $\chi(g)^n=1$ which implies $|\chi(g)|=1$ and $\overline{\chi(g)}=\chi(g)^{-1}$.
Thus a finite cyclic group $\ZZ_e$ has $e$ distinct linear characters given by $\chi_a(b):=\omega^{ab}$ for all $a,b\in \ZZ_e$, where $\omega:=\exp(2\pi i/e)$ is a primitive $e$th root of unity.
Moreover, if $G_1$ and $G_2$ are finite abelian groups then $\chi$ is a linear character of $G_1\times G_2$ if and only if $\chi=\chi_1\cdot\chi_2$ for some linear characters $\chi_1\in G_1^*$ and $\chi_2\in G_2^*$.
Therefore $G^* \cong G$ for any finite abelian group $G$; in particular, $|G^*|=|G|=\dim(\CC^G)$. 
If $\chi\in G^*$ then $\langle\chi,\chi\rangle =1$ since $\chi(g)\overline{\chi(g)} = |\chi(g)|=1$ for all $g\in G$.
For distinct $\chi, \chi'\in G^*$, we have $\chi(h)\ne\chi'(h)$ for some $h\in G$ and thus
\[ \langle \chi, \chi' \rangle = \frac{1}{|G|} \sum_{g\in G} \chi(g) \chi'(g)^{-1} 
=  \frac{1}{|G|} \sum_{g\in G} \chi(gh) \chi'(gh)^{-1}
= \chi(h)\chi'(h)^{-1} \langle \chi, \chi' \rangle \]
which implies that $\langle \chi, \chi' \rangle = 0$.
Therefore $G^*$ is an orthonormal basis for $\CC^G$.
\end{proof}

Let $G$ be a finite abelian group expressed additively, and let $S$ be a subset of $G-\{0\}$ such that $s\in S\Rightarrow -s\in S$.
The \emph{Cayley graph} $\Gamma=\Gamma(G,S)$ of $G$ with respect to $S$ has vertex set $X=G$ and edge set $E=\{ xy: y-x \in S\}$.
The eigenvalues and eigenvectors of $\Gamma$ are those of its \emph{adjacency matrix} $A=[a_{xy}]_{x,y\in X}$, where $a_{xy}$ is one if $xy\in E$ or zero otherwise. 
An \emph{eigenbasis} of $\Gamma$ is a basis for the vector space $\CC^X$ consisting of eigenvectors of $\Gamma$.
An explicit construction of such a basis is well known in the abelian case~\cite[Exercise~11.8]{Lovasz79};
for the nonabelian case, see, for example, Babai~\cite{Babai}, Brouwer--Haemers~\cite[Proposition~6.3.1]{Spectra}, and Lov\'asz~\cite{Lovasz75}.

\begin{theorem}\label{thm:Cayley}
For any Cayley graph $\Gamma=\Gamma(X,S)$ of a finite abelian group $X$, the linear characters of $X$ form an eigenbasis of $\Gamma$ with each linear character $\chi$ corresponding to the eigenvalue $\chi(S):=\sum_{s\in S}\chi(s)$.
\end{theorem}

\begin{proof}
By Theorem~\ref{thm:character}, the linear characters of $X$ form a basis for the space $\CC^X$.
Each linear character $\chi$ of $X$ is an eigenvector of $\Gamma$ corresponding to the eigenvalue $\chi(S)$ since
\[ (A \chi)(x) = \sum_{y-x\in S} \chi(y) = \sum_{s\in S} \chi(s) \chi(x) = \chi(S)\chi(x) \quad\text{for all } x\in X. \qedhere \]
\end{proof}

Since any Cayley graph $\Gamma(X,S)$ of a finite abelian group $X$ has an eigenbasis $X^*$ consisting of all linear characters of $X$, we can define the \emph{Norton product} $\chi\star\chi'$ of two linear characters $\chi$ and $\chi'$ in the same eigenspace of $\Gamma$ by projecting the entry-wise product $\chi\cdot\chi'$ back to this eigenspace.
We provide a formula for this product below.

\begin{theorem}\label{thm:Norton}
For any Cayley graph $\Gamma(X,S)$ of a finite abelian group $X$, if two linear characters $\chi$ and $\chi'$ of $X$ correspond to the same eigenvalue $\chi(S)=\chi'(S)$ then
\[ \chi\star \chi' = 
\begin{cases}
\chi\cdot\chi' & \text{if } (\chi\cdot\chi')(S) = \chi(S) \\
0, & \text{otherwise}.
\end{cases} \]
\end{theorem}

\begin{proof}
Given linear characters $\chi$ and $\chi'$ of $X$ with $\chi(S)=\chi'(S)$, the entry-wise product $\chi\cdot\chi'$ is still a linear character of $X$ with the corresponding eigenvalue $(\chi\cdot\chi')(S)$.
The projection onto the eigenspace containing $\chi$ and $\chi'$ fixes $\chi\cdot\chi'$ if $(\chi\cdot\chi')(S)=\chi(S)$ or annihilates it  otherwise.
\end{proof}

In the remainder of this paper we elaborate the aforementioned linear character approach to the Norton algebras in the context of the Hamming graphs, halved/folded cubes, and bilinear forms graphs, as these graphs are simultaneously distance regular graphs and Cayley graphs of finite abelian groups.
Even though their eigenspaces can be realized over $\RR$, we allow extension of scalars to the complex field $\CC$ so that we can apply Theorem~\ref{thm:Cayley} and Theorem~\ref{thm:Norton} to obtain linear character bases and explicit formulas of the Norton product for all eigenspaces.
When there are multiple graphs involved, we let $V_i(\Gamma)$ denote the $i$th eigenspace of $\Gamma$ and the corresponding Norton algebra.
 
\section{Hamming graphs}\label{sec:Hamming}

In this section we study the Norton algebras of the Hamming graphs.
Given integers $n\ge1$ and $e\ge2$, the finite abelian group $X=\ZZ_e^n$ consists of all words of length $n$ on the alphabet $\ZZ_e=\{0,1,\ldots,e-1\}$ with addition performed entry-wise modulo $e$.
We can write an element $x\in X$ as either a word $x=x_1\cdots x_n$ or a function $x:[n]\to\ZZ_e$, where $[n]:=\{1,2,\ldots,n\}$.
If $x(1)=\cdots=x(n)=c$ for some constant $c\in\CC$ then we write $x=c$.
The \emph{support} of $x$ is $\supp(x):=\{j\in[n]: x(j)\ne0\}$.
Let $X_i$ denote the set of all $x\in X$ with $|\supp(x)|=i$.

The \emph{Hamming graph} $H(n,e)$ is the Cayley graph $\Gamma(\ZZ_e^n,X_1)$, which is a distance regular graph of diameter $d=n$.
For $i=0,1,\ldots,n$, the $i$th eigenvalue of $H(n,e)$ and its multiplicity are~\cite[\S9.2]{DistReg1}
\[ \theta_i=(n-i)e-n \qand \dim(V_i)=\binom{n}{i}(e-1)^i.\]

Maldonado and Penazzi~\cite{NortonAlgebra} showed that the Norton product on the eigenspace $V_1$ of the Hamming graph $H(n,2)$ is constantly zero. 
For $e\ge3$, we showed in previous work~\cite{Huang20} that the Norton algebra $V_1(H(1,e))$ has a spanning set $\{\bar v_1,\ldots,\bar v_{e}\}$ such that 
\begin{equation}\label{eq:V1}
\bar v_i\star \bar v_i= \bar v_i \qand \bar v_i\star \bar v_j = -(\bar v_i+\bar v_j)/(e-2), \quad 1\le i\ne j\le e 
\end{equation}
and the direct product of $n$ copies this algebra is isomorphic to the Norton algebra $V_1(H(n,e))$.
Now using linear characters we can determine all Norton algebras of $H(n,e)$.

\subsection{Basis and Norton product}
For every $u\in X$ we define a linear character $\chi_u$ of $X$ by
\begin{equation}\label{eq:chi}
\chi_u(x):= \prod_{j\in[n]} \omega^{u(j)x(j)} = \omega^{\sum_{j\in[n]} u(j)x(j)} 
\quad \text{for all } x \in X.
\end{equation}
Here $\omega:=\exp(2\pi i/e)\in\CC$ is a primitive $e$th root of unity, which satisfies the following: 
\begin{equation}\label{eq:root}
1+\omega^j+\omega^{2j}+\cdots+\omega^{(e-1)j} = 
\begin{cases}
e & \text{if } j=0, \\
0 & \text{if } j=1,\ldots,e-1.
\end{cases} \
\end{equation}
The proof of the above identity is an easy exercise.
We are ready to provide our first main result on the Hamming graphs.

\begin{theorem}\label{thm:NortonHamming}
For $i=0,1,\ldots,n$, the (complex) eigenspace $V_i$ of the Hamming graph $H(n,e)$ has a basis $\left\{\chi_u : u\in X_i \right\}$
such that for all $u,v\in X_i$,
\[ \chi_u \star \chi_v =
\begin{cases}
\chi_{u+v} & \text{if } u+v\in X_i \\
0 & \text{otherwise}.
\end{cases} \]
\end{theorem}

\begin{proof}
By Theorem~\ref{thm:character} and its proof, $\{\chi_u:u\in X\}$ is a complete set of distinct linear characters of the abelian group $X=\ZZ_e^n$ and thus a basis for $\CC^X$.
For each $u\in X$, the linear character $\chi_u$ is an eigenvector of $H(n,e)$ corresponding to the eigenvalue $\chi_u(X_1)$ by Theorem~\ref{thm:Cayley}.

To compute the eigenvalue $\chi_u(X_1)$, suppose $u\in X_i$ and let $x\in X_1$.
Then we have $x(j)\ne0$ for a unique $j\in[n]$ and thus $\chi_u(x) = \omega^{u(j)x(j)}$.
If $j\notin\supp(u)$ then $u(j)=0$ and $\chi_u(x) =1$.
If $j\in \supp(u)$ then $\omega^{u(j)}$ is a nontrivial $e$th root of unity and the sum of $\chi_u(x) = \omega^{u(j)x(j)}$ as $x(j)$ runs through $\{1,2,\ldots,e-1\}$ with $j$ fixed is $-1$, thanks to Equation~\eqref{eq:root}.
Thus
\[ \chi_u(X_1) = \sum_{j\in[n]} \sum_{k\in\ZZ_e-\{0\}} \omega^{u(j)k}
= (n-i)(e-1)+i(-1)=\theta_i.
\]
This implies that $\{\chi_u: u\in X_i\}$ is a basis for $V_i$.
For any $u,v\in X_i$ we have
\[ (\chi_u \cdot \chi_v )(x) = \prod_{j\in[n]} \omega^{(u(j)+v(j))x(j)} \quad \text{for all } x\in X.\]
Thus $\chi_u\cdot \chi_v = \chi_{u+v} \in V_j$ where $j=|\supp(u+v)|$.
The projection of $\chi_{u+v}$ to $V_i$ is either itself if $j=i$ or zero otherwise.
\end{proof}

\begin{example}\label{example:H23}
The complex eigenbasis for the Hamming graph $H(2,3)$ given in Theorem~\ref{thm:NortonHamming} consists of the rows of the matrix below, which are indexed by the vertices (written as words).
\begin{center}
\begin{tabular}{C|CCCCCCCCC}
X & 00 & 10 & 01 & 20 & 11 & 02 & 21 & 12 & 22  \\
\hline
\chi_{00} & 1 & 1 & 1 & 1 & 1 & 1 & 1 & 1 & 1\\
\chi_{10} & 1 & \omega & 1 & \omega^2 & \omega & 1 & \omega^2 &\omega & \omega^2 \\
\chi_{20} & 1 & \omega^2 & 1 & \omega & \omega^2 & 1 & \omega & \omega^2 & \omega \\
\chi_{01} & 1 & 1 & \omega & 1 & \omega & \omega^2 & \omega & \omega^2 & \omega^2 \\
\chi_{02} & 1 & 1 & \omega^2 & 1 & \omega^2 & \omega & \omega^2 &\omega &\omega \\
\chi_{11} & 1 & \omega &\omega &\omega^2 & \omega^2 & \omega^2 & 1 & 1 & \omega \\
\chi_{21} & 1 & \omega^2 & \omega & \omega & 1 & \omega^2 & \omega^2 & \omega & 1 \\
\chi_{12} & 1 & \omega & \omega^2 & \omega^2 & 1 & \omega &\omega &\omega^2 & 1 \\
\chi_{22} & 1 & \omega^2 & \omega^2 & \omega &\omega &\omega &1& 1& \omega^2  
\end{tabular}
\end{center}
The first row spans $V_0$ with $\chi_{00}\star\chi_{00} = \chi_{00}$, the next four rows span $V_1$, and the last four rows span $V_2$.
The Norton algebras $V_1$ and $V_2$ are isomorphic to each other by the following charts (a coincidence between Corollary~\ref{cor:V1Hn} and Proposition~\ref{prop:VnHn}).
\begin{center}
\begin{tabular}{c|cccc}
$\star$ & $\chi_{01}$ & $\chi_{02}$ & $\chi_{10}$ & $\chi_{20}$ \\
\hline
$\chi_{01}$ & $\chi_{02}$ & $0$ & $0$ & $0$ \\
$\chi_{02}$ & $0$ & $\chi_{01}$ & $0$ & $0$ \\
$\chi_{10}$ & $0$ & $0$ & $\chi_{20}$ & $0$ \\
$\chi_{20}$ & $0$ & $0$ & $0$ & $\chi_{10}$
\end{tabular} \qquad
\begin{tabular}{c|cccc}
$\star$ & $\chi_{11}$ & $\chi_{12}$ & $\chi_{21}$ & $\chi_{22}$ \\
\hline
$\chi_{11}$ & $\chi_{22}$ & $0$ & $0$ & $0$ \\
$\chi_{12}$ & $0$ & $\chi_{21}$ & $0$ & $0$ \\
$\chi_{21}$ & $0$ & $0$ & $\chi_{12}$ & $0$ \\
$\chi_{22}$ & $0$ & $0$ & $0$ & $\chi_{11}$
\end{tabular} 
\end{center}
\end{example}

\begin{remark}
Although the basis given in Theorem~\ref{thm:NortonHamming} consists of complex vectors, we can obtain a real basis by taking real and imaginary parts of the vectors. 
Let $X_i^0:= \{u\in X_i: 2u=0\}$ and 
let $X_i^+$ and $X_i^-$ be the sets of all $u\in X_i-X_i^0$ with $u(j)>e/2$ or $u(j)<e/2$, respectively, where $j$ is the smallest integer such that $u(j)\ne e/2$.
For each $u\in X$, define $\xi_u: X\to \RR$ by 
\[ \xi_u(x) :=
\begin{cases}
\cos\left( \sum_{j\in[n]} \frac{2\pi u(j)x(j)}e \right) & \text{if } u\in X_i^+ \cup X_i^0 \\
\sin\left( \sum_{j\in[n]} \frac{2\pi u(j)x(j)}e \right) & \text{if } u\in X_i^-
\end{cases} \]
for all $x\in X$.
One sees that for $i=0,1,\ldots,d=n$, the eigenspace $V_i$ of $H(n,e)$ has a real basis
$\{ \xi_u: u\in X_i\}$.
Below is an example for the Hamming graph $H(2,3)$, which can be further normalized to an eigenbasis over $\ZZ$.
\begin{center}
\begin{tabular}{C|CCCCCCCCC}
X & 00 & 10 & 01 & 20 & 11 & 02 & 21 & 12 & 22  \\
\hline
\chi_{00} & 1 & 1 & 1 & 1 & 1 & 1 & 1 & 1 & 1\\
\chi_{10} & 0 & \sqrt{3}/2 & 0 & -\sqrt{3}/2 & \sqrt{3}/2 & 0 & -\sqrt{3}/2 &\sqrt{3}/2 & -\sqrt{3}/2 \\
\chi_{20} & 1 & -1/2 & 1 & -1/2 & -1/2 & 1 & -1/2 & -1/2 & -1/2 \\
\chi_{01} & 0 & 0 & \sqrt{3}/2 & 0 & \sqrt{3}/2 & -\sqrt{3}/2 & \sqrt{3}/2 & -\sqrt{3}/2 & -\sqrt{3}/2 \\
\chi_{02} & 1 & 1 & -1/2 & 1 & -1/2 & -1/2 & -1/2 &-1/2 &-1/2 \\
\chi_{11} & 0 & \sqrt{3}/2 &\sqrt{3}/2 &-\sqrt{3}/2 & -\sqrt{3}/2 & -\sqrt{3}/2 & 0 & 0 & \sqrt{3}/2 \\
\chi_{21} & 1 & -1/2 & -1/2 & -1/2 & 1 & -1/2 & -1/2 & -1/2 & 1 \\
\chi_{12} & 0 & \sqrt{3}/2 & -\sqrt{3}/2 & -\sqrt{3}/2 & 0 & \sqrt{3}/2 &\sqrt{3}/2 &-\sqrt{3}/2 & 0 \\
\chi_{22} & 1 & -1/2 & -1/2 & -1/2 &-1/2 &-1/2 &1& 1& -1/2  
\end{tabular}
\end{center}
Recently, there has been work on the minimum cardinality of the support of eigenfunctions of a Hamming graph; see for example, Valyuzhenich and Vorob'ev~\cite{Supp}.
The linear characters may provide another possible approach to such problems.
\end{remark}

\subsection{Algebra structure}
Theorem~\ref{thm:NortonHamming} leads to the following result on the structure of the Norton algebras of $H(n,e)$.

\begin{corollary}\label{cor:Hamming}
For $i=0,1,\ldots,n$, we have a direct sum decomposition of vector spaces
\[ V_i(H(n,e)) = \bigoplus_{S\subseteq [n]: |S|=i} V_S \]
where $V_S:= \mathrm{span} \{ \chi_u: u\in X_i,\ \supp(u)=S \}$ is isomorphic to $V_i(H(i,e))$ as an algebra.
This direct sum becomes a direct product of algebras if $i=0,1,n$ or if $e=2$ and either $i>\lfloor 2n/3 \rfloor$ or $i$ is odd.
\end{corollary}

\begin{proof}
Using the basis provided in Theorem~\ref{thm:NortonHamming}, one sees that the eigenspace $V_i(H(n,e))$ is the direct sum of subspaces $V_S$ for all $i$-sets $S\subseteq[n]$.
By the formula for the Norton product on $V_i(H(n,e))$ in Theorem~\ref{thm:NortonHamming}, each direct summand $V_S$ is a subalgebra isomorphic to $V_i(H(i,e))$ by sending $\chi_u$ to $\chi_{\bar u}$ for all $u\in X_i$ with $\supp(u)=S$, where $\bar u$ is obtained from $u$ by deleting all zero entries.
Finally, the above direct sum becomes a direct product of algebras in the following cases.
\begin{itemize}
\item
If $i=0$ or $i=n$ then this direct sum has only one summand.
\item
If $i=1$ then $\chi_u\star\chi_v=0$ whenever $u,v\in X_1$ with $\supp(u)\ne\supp(v)$.
\item
If $e=2$ and either $i>\lfloor 2n/3 \rfloor$ or $i$ is odd then $\chi_u\star\chi_v=0$ for all $u,v\in X_i$ by Proposition~\ref{prop:zero}, which will be provided later.
\qedhere
\end{itemize}

\end{proof}

The above corollary includes a previous result as a special case.

\begin{corollary}[\cite{Huang20}]\label{cor:V1Hn}
The Norton algebra $V_1(H(n,e))$ is isomorphic to the direct product of $n$ copies of the Norton algebra $V_1(H(1,e))$.
\end{corollary}

We have another similar result on the Norton algebra $V_n(H(n,3))$.

\begin{lemma}\label{lem:Vu}
For any $u\in X_i$ with $\supp(u)=S$ and $u(S)\subseteq \ZZ_e^\times$, the span $V_u$ of $\{ \chi_{ku}: k\in[e-1] \}$ is a subalgebra of the subalgebra $V_S$ of $V_i(H(n,e))$ satisfying $V_u \cong V_1(H(1,e))$.
\end{lemma}

\begin{proof}
We have $\supp(ku)=S$ for all $k\in[e-1]$ since $u(S)\subseteq\ZZ_e^\times$.
Thus $V_u$ is a subalgebra of $V_S$ and is isomorphic to the Norton algebra $V_1(H(1,e))$ via $\chi_{ku}\mapsto \chi_k$.
\end{proof}

\begin{proposition}\label{prop:VnHn}
The Norton algebra $V_n(H(n,3))$ is isomorphic to the direct product of $2^{n-1}$ copies of the Norton algebra $V_1(H(1,3))$. 
\end{proposition}

\begin{proof}
By Lemma~\ref{lem:Vu}, there is a subalgebra $V_u\cong V_1(H(1,3))$ spanned by $\chi_u$ and $\chi_{2u}$ for all $u\in X_n$.
Let $u$ and $v$ be distinct elements in $X_n$.
We have $\{u(j), v(j)\} = \{1,2\}$ for some $j\in[n]$ and thus $u+v=0$.
This shows that $V_u$ is orthogonal to $V_v$ if $u\ne 2v$, and it is clear that $V_u=V_v$ if $u=2v$.
Thus $V_n(H(n,3))$ is isomorphic to the direct product of $2^{n-1}$ copies of $V_1(H(1,3))$.
\end{proof}

\begin{example}
The Norton algebra $V_2(H(2,3))$ has a basis $\{\chi_{11}, \chi_{12}, \chi_{21}, \chi_{22}\}$.
The span of $\chi_{11}$ and $\chi_{22}$ is isomorphic to $V_1(H(1,3))$, and so is the span of $\chi_{12}$ and $\chi_{21}$. 
These two copies of $V_1(H(1,3))$ are orthogonal in $V_2(H(2,3)$ and form a direct product.
Similarly, the Norton algebra $V_3(H(3,3))$ is the direct product of four subalgebras, each of which is isomorphic to $V_1(H(1,3))$. 
\end{example}

We also observe that there is no identity for the Norton product $\star$ on $V_i(H(n,e))$ unless $i=0$.

\begin{proposition}
The Norton algebra $V_i(H(n,e))$ is unital if and only if $i=0$. 
\end{proposition}

\begin{proof}
Any element of $V_i(H(n,e))$ can be written as $\sum_{u\in X_i} c_u \chi_u$. 
For any $v\in X_i$ we have
\[ \sum_{u\in X_i} c_u \chi_u\star \chi_v = \sum_{u\in X_i:\ u+v \in X_i} c_u \chi_{u+v} = \chi_v \]
if and only if $0\in X_i$ and $c_0=1$. 
The result follows immediately.
\end{proof}

\subsection{Hypercube}

Finally, we focus on the case $e=2$.
In this case the Hamming graph $H(n,2)$ is known as the \emph{hypercube} $Q_n$.
The basis given in Theorem~\ref{thm:NortonHamming} for each eigenspace of $Q_n$ is real and actually over $\{0,1,-1\}$ since $\omega=-1$ when $e=2$.
Moreover, each element $u\in X_i$ is uniquely determined by its support $S=\supp(u)$, which is an $i$-subset of $[n]$, and the linear character $\chi_u = \chi_S$ is given by 
\[ \chi_S(x):= \prod_{j\in S} (-1)^{x(j)} \quad\text{for all }x\in X.\]
Using the \emph{symmetric difference} $S\triangle T:= (S-T)\cup(T-S)$ of two sets $S$ and $T$ we can rephrase Theorem~\ref{thm:NortonHamming} for the hypercube $Q_n$ below.

\begin{corollary}\label{cor:hypercube}
For $i=0,1,\ldots,n$, there exists a basis $\{\chi_S: S\subseteq[n],\ |S|=i\}$ for the eigenspace $V_i$ of the hypercube graph $Q_n$ such that for all $S,T\subseteq[n]$ with $|S|=|T|=i$,
\[ \chi_S \star \chi_T =
\begin{cases}
\chi_{S\triangle T} & \text{if } |S\triangle T|=i \\
0 & \text{otherwise}.
\end{cases} \]
\end{corollary}
\begin{proof}
This follows immediately from Theorem~\ref{thm:NortonHamming} with $e=2$.
\end{proof}

\begin{example}\label{example:H32}
The Norton algebra $V_2(Q_3)$ has a basis $\{ \chi_R, \chi_S, \chi_T \}$, where $R=\{1,2\}$, $S=\{1,3\}$, $T=\{2,3\}$.  
We have $\chi_R\star \chi_R=\chi_S \star \chi_S=\chi_T\star\chi_T=0$, $\chi_R \star\chi_S = \chi_T$, $\chi_S\star \chi_T = \chi_R$, and $\chi_T\star\chi_R = \chi_S$.
This agrees with the cross product on the three-dimensional space, except for the anticommutativity.
\end{example}

\begin{lemma}\label{lem:Delta}
There exist $i$-sets $S,T\subseteq [n]$ such that $|S\triangle T|=j$ if and only if $0\le j\le \min\{2i,2(n-i)\}$ and $j$ is even.
\end{lemma}

\begin{proof}
Let $S$ and $T$ be $i$-subsets of $[n]$ with $|S\triangle T|=j$.
We have
\[ \begin{aligned}
|S\cap T| &= (|S|+|T|-|S\triangle T|)/2 = (2i-j)/2 \ge0 \qand \\
|S\cup T| &= |S\triangle T|+|S\cap T|=(2i+j)/2 \le n.
\end{aligned} \]
Thus we have $0\le j\le \min\{2i,2(n-i)\}$ and $j$ must be even 

Conversely, let $j$ be an even integer satisfying $0\le j\le \min\{2i,2(n-i)\}$.
Then $S=\{1,2,\ldots,i\}$ and $T=\{1,2,\ldots,(2i-j)/2, i+1,\ldots,(2i+j)/2\}$ are $i$-subsets of $[n]$ with	 $|S\triangle T|=j$.
\end{proof}

\begin{proposition}\label{prop:zero}
If $i> \lfloor 2n/3 \rfloor$ or $i$ is odd then the Norton algebra $V_i(Q_n)$ has a zero product.
\end{proposition}

\begin{proof}
By Lemma~\ref{lem:Delta}, there exist $i$-sets $S,T\subseteq[n]$ such that $|S\triangle T|=i$ if and only if $i\le \lfloor 2n/3 \rfloor$ and $i$ is even.
Thus if $i> \lfloor 2n/3 \rfloor$ or $i$ is odd then $\chi_S\star\chi_T=0$ for all $i$-sets $S,T\subseteq[n]$.
\end{proof}

\section{Automorphisms}\label{sec:Aut}
Recall that an \emph{automorphism} of an algebra $(V,\star)$ is an automorphism $\phi$ of the underlying vector space $V$ such that $\phi(u\star v)=\phi(u)\star\phi(v)$ for all $u,v\in V$.
Note that we do not assume an algebra is associative or unital.
In this section we study the automorphisms of the Norton algebras of the Hamming graph $H(n,e)$.
We begin with some special cases.

\begin{proposition}\label{prop:AutV0Hn}
The automorphism group of the Norton algebra $V_i(H(n,e))$ is the trivial group if $i=0$ or the full general linear group of the underlying vector space if $e=2$ and either $i>\lfloor 2n/3 \rfloor$ or $i$ is odd.
\end{proposition}

\begin{proof}
The Norton algebra $V_i(H(n,e))$ is spanned by an idempotent $\chi_{0}$ if $i=0$, or has a zero product if $e=2$ and either $i>\lfloor 2n/3 \rfloor$ or $i$ is odd by Proposition~\ref{prop:zero}.
The result follows.
\end{proof}

To deal with the remaining cases, we review some basic concepts from group theory.
If there is a homomorphism $\phi: H\to\Aut(N)$ from a group $H$ to the automorphism group of a group $N$, then the \emph{semidirect product} $N\rtimes H$ is the Cartesian product $N\times H$ endowed with the operation $(n_1,h_1)(n_2,h_2)=(n_1 \phi(h_1)(n_2), h_1h_2)$ for all $n_1,n_2\in N$ and $h_1,h_2\in H$.
In particular, the semidirect product $\ZZ_e\rtimes \ZZ_e^\times$ of the group $\ZZ_e$ and its multiplicative group $\ZZ_e^\times$ is the Cartesian product $\ZZ_e \times \ZZ_e^\times$ equipped with an operation defined by $(a,b)(a',b'):=(a+b^{-1}a',bb')$ for all $a,a'\in\ZZ_e$ and $b,b'\in \ZZ_e^\times$.
The identity element of this group is $(0,1)$, where $0$ and $1$ are the identity elements of $\ZZ_e$ and $\ZZ_e^\times$, respectively.

Next, the \emph{symmetric group} $\SS_n$ consists of all permutations on the set $[n]$, and we let $\id$ denote its identity.
The \emph{wreath product} $G\wr\SS_n$ of a group $G$ with $\SS_n$ is the semidirect product $G^n \rtimes \SS_n$, where $G^n$ is the direct product of $n$ copies of $G$ and $\SS_n$ permutes these $n$ copies.
More precisely, the elements of $G\wr\SS_n$ are of the form $(g,\sigma)$ with $g\in G^n$ and $\sigma\in\SS_n$, and with $h_j$ denoting the $j$th component of an $n$-tuple $h\in G^n$, the operation of $G\wr\SS_n$ is defined by $(g,\sigma)(g',\sigma') = (g\sigma(g'),\sigma\sigma')$ where $\sigma(g') = (g'_{\sigma(1)},\ldots,g'_{\sigma(n)})$. 

The automorphism group of the Hamming graph $H(n,e)$ is isomorphic to the wreath product $\SS_e \wr \SS_n$~\cite[Theorem~9.2.1]{DistReg1}, which acts on $X=\ZZ_e^n$ in a natural way.
However, not all automorphisms of $H(n,e)$ induce automorphisms of the Norton algebra $V_i(H(n,e))$ nor can they give all automorphisms of $V_i(H(n,e))$.
This can be seen in our next result, which involves the group $(\ZZ_e\times\ZZ_e^\times)\rtimes\SS_n$.
The elements in this group are of the form $(a,b,\sigma)$ with $a\in\ZZ_e^n$, $b\in(\ZZ_e^\times)^n$, and $\sigma\in\SS_n$.
If $(a,b,\sigma)$ and $(a',b',\sigma')$ are in this group then  
\begin{equation}\label{eq:wr}
(a,b,\sigma)(a',b',\sigma') = (a+b^{-1}\cdot\sigma(a'), b\cdot\sigma(b'), \sigma\sigma').
\end{equation}
Here the dot ``$\cdot$'' denotes the entry-wise product on $\ZZ_e^n$, which makes $\ZZ_e^n$ become a ring with $(\ZZ_e^\times)^n$ as its group of units, and $\sigma\sigma'$ is the composition of permutations.

\begin{theorem}\label{thm:AutHamming}
Each element $\phi=(a,b,\sigma)\in(\ZZ_e\rtimes \ZZ_e^\times)\wr\SS_n$ induces an automorphism of the Norton algebra $V_i(H(n,e))$ by sending $\chi_u$ to $\chi_a(b\cdot\sigma(u))\chi_{b\cdot\sigma(u)}$ for all $u\in X_i$.
Such automorphisms form a group isomorphic to $(\ZZ_e\rtimes\ZZ_e^\times) \wr \SS_n$ if $e\ge3$ and $i\ge1$ or if $e=2$ and $1\le i<n$ is odd, but isomorphic to $(\ZZ_2\wr \SS_n) / \{ (0,1,\id), (1,1,\id)\}$ if $e=2$ and $1\le i<n$ is even.
\end{theorem}

\begin{proof}
Let $\phi=(a,b,\sigma)\in(\ZZ_e\rtimes \ZZ_e^\times)\wr\SS_n$.
Then $u$ and $b\cdot\sigma(u)$ have the same support for any $u\in X_i$ since $b\in(\ZZ_e^\times)^n$.
If $u,v\in X_i$ with $u+v\in X_i$ then $b\cdot\sigma(u+v) = b\cdot\sigma(u)+b\cdot\sigma(v)$ and 
\[ \phi(\chi_u\star\chi_v) = \chi_a(b\cdot\sigma(u+v))\chi_{b\cdot \sigma(u+v)} = \chi_a(b\cdot\sigma(u))\chi_a(b\cdot\sigma(v))\chi_{b\cdot\sigma(u) + b\cdot\sigma(v)} = \phi(\chi_u)\star\phi(\chi_v). \] 
Similarly, if $u,v\in X_i$ with $u+v\notin X_i$ then 
\[ \phi(\chi_u\star\chi_v) = \phi(0) = 0 = \chi_a(b\cdot\sigma(u))\chi_{b\cdot\sigma(u)}\star \chi_a(b\cdot\sigma(v))\chi_{b\cdot\sigma(v)} = \phi(\chi_u)\star\phi(\chi_v). \] 
Therefore $\phi$ induces an automorphism of the Norton algebra $V_i(H(n,e))$.

If $\phi'=(a',b',\sigma') \in(\ZZ_e\rtimes \ZZ_e^\times)\wr\SS_n$ then for any $u\in X_i$ we have
\[ \phi(\phi'(\chi_u)) = \phi(\chi_{a'}(b'\cdot\sigma'(u))\chi_{b'\cdot\sigma'(u)}) 
= \chi_{a}(b\cdot\sigma(b')\cdot\sigma\sigma'(u)) \chi_{a'}(b'\cdot\sigma'(u)) \chi_{b\cdot\sigma(b')\cdot\sigma\sigma'(u)}. \]
On the other hand, we have $\phi\phi' = (a+b^{-1}\cdot\sigma(a'), b\cdot\sigma(b'), \sigma\sigma')$ by Equation~\eqref{eq:wr} and thus
\[ (\phi\phi')(\chi_u) 
= \chi_{a}(b\cdot\sigma(b')\cdot\sigma\sigma'(u)) \chi_{b^{-1}\cdot\sigma(a')}(b\cdot\sigma(b')\cdot\sigma\sigma'(u)) \chi_{b\cdot\sigma(b')\cdot\sigma\sigma'(u)} \]
Then we obtain $\phi(\phi'(\chi_u))=(\phi\phi')(u)$ as Equation~\eqref{eq:chi} implies
\[ \chi_{b^{-1}\cdot\sigma(a')} (b\cdot\sigma(b')\cdot\sigma\sigma'(u)) = \chi_{\sigma(a')} (\sigma(b'\cdot\sigma'(u)) = \chi_{a'}(b'\cdot\sigma'(u)).\]

It follows that we have a homomorphism from $(\ZZ_e\rtimes \ZZ_e^\times)\wr\SS_n$ to the automorphism group of $V_i(H(n,e))$, and we need to show that its kernel is trivial if $e\ge3$ and $i\ge1$ or if $e=2$ and $1\le i<n$ is odd, or equals $\{ (0,1,\id), (1,1,\id)\}$ if $e=2$ and $1\le i<n$ is even.
To this end, suppose that $\phi(\chi_u)=\chi_u$, i.e., $b\cdot\sigma(u)=u$ and $\chi_a(b\cdot\sigma(u))=\chi_a(u)=1$ for all $u\in X_i$.
 
If $1\le i<n$ then $\SS_n$ acts faithfully on $i$-subsets of $[n]$ and thus taking the support of both sides of the equality $b\cdot\sigma(u)=u$ gives $\sigma=\id$.
If $i=n$ and $e\ge3$ then we also have $\sigma=\id$ since if $\sigma(j)=k\ne j$ for some $j\in[n]$ then $b\cdot\sigma(u)\ne u$ for some $u\in X_i$ with $u(j)=1$ and $u(k)\ne b_k$.

Next, if $\sigma=\id$ then $b\cdot\sigma(u)=b\cdot u=u$ for all $u\in X_i$ and thus $b=1$.

Finally, we consider the condition that $\chi_a(u)=1$ for all $u\in X_i$.
\begin{itemize}
\item
If $e\ge3$ then $a=0$ as we can obtain $v\in X_i$ from $u$ by changing the $j$th entry to a different nonzero number and then $\chi_a(u)=\chi_a(v)$ implies $a(j)=0$ for any $j\in[n]$.
\item
Assume $e=2$ and $1\le i<n$.
For any distinct $j,k\in[n]$, there exists $u\in X_i$ such that $u(j)=1$ and $u(k)=0$.
We obtain $v\in X_i$ from $u$ by switching the $j$th and $k$th entries.
Then $\chi_a(u)=\chi_a(v)$ implies $a(j)=a(k)$.
Thus $a$ is either $0$ or $1$, and the latter is possible if and only if $i$ is even in order to have $\chi_{1}(u)=1$ for all $u\in X_i$. \qedhere
\end{itemize} 
\end{proof}

When $e=2$ the group in Theorem~\ref{thm:AutHamming} becomes $\ZZ_2\wr\SS_n$.
This group is often interpreted as the \emph{hyperoctahedral group} $\SS_n^B$, which consists of bijections $f$ on $\{\pm1,\ldots,\pm n\}$ with $f(-j)=-f(j)$ for all $j\in[n]$.
The group $\SS_n^B$ is the Coxeter group of type $B_n$ and its elements are called \emph{signed permutations} of $[n]$ as they are determined by their values on $[n]$.
We can write an element in $\SS_n^B$ as $f=(\sigma,\epsilon)$ for unique $\sigma\in\SS_n$ and $\epsilon:[n]\to\{\pm1\}$ such that $f(j)= \epsilon(\sigma(j))\sigma(j) =(\epsilon\sigma)(j)\sigma(j)$ for all $j\in [n]$, where $\epsilon\sigma$ is the composition of $\epsilon$ and $\sigma$.
If $f'=(\sigma', \epsilon')\in\SS_n^B$ then $ff' = (\epsilon\cdot \epsilon'\sigma^{-1}, \sigma\sigma')$, where the dot ``$\cdot$'' is the entry-wise product, since
\begin{equation}\label{eq:SB}
f(f'(j)) = f( \epsilon'\sigma'(j)\sigma'(j) ) = \epsilon\sigma\sigma'(j) \epsilon'\sigma'(j) \sigma\sigma'(j) \quad\text{for all } j\in[n]. 
\end{equation}
Comparing this with the operation in the group $\ZZ_2\wr\SS_n$ as a special case ($b=b'=1$) of Equation~\eqref{eq:wr}, we have a group isomorphism $\ZZ_2\wr\SS_n \cong \SS_n^B$ by sending $a\in\ZZ_2^n$ to $\epsilon:[n]\to\{\pm1\}$ with $\epsilon(j) := (-1)^{a(j)}$ for all $j\in[n]$.
This isomorphism takes the subgroup $\{(0,\id), (1,\id)\}$ of $\ZZ_2\wr\SS_n$ to the center $\{\pm1\}$ of $\SS_n^B$, where
by abuse of notation we write $f=c$ if there exists a constant $c$ such that $f(j)=c$ for all $j\in[n]$. 
Then we can rephrase Theorem~\ref{thm:AutHamming} in the case $e=2$ as follows.

\begin{corollary}\label{cor:AutH2}
Every signed permutation in $\SS_n^B$ induces an automorphism of the Norton algebra $V_i(Q_n)$ by sending $\chi_S$ to $\epsilon(\sigma(S))\chi_{\sigma(S)}$, where $\epsilon(T):=\prod_{j\in T}\epsilon(j)$ for all $T\subseteq[n]$.
Such automorphisms form a group isomorphic to $\SS_n^B$ if $1\le i<n$ is odd or $\SS_n^B/\{\pm1\}$ if $1\le i<n$ is even.
\end{corollary}

In general, the automorphism group of the Norton algebra $V_i(H(n,e))$ does not equal the subgroup given in Theorem~\ref{thm:AutHamming}, as shown in Example~\ref{example:AutHamming} below.
To better understand it, we try to construct all of the idempotents in $V_i(H(n,e))$, as any algebra automorphism must permute the idempotents.
Here an \emph{idempotent} is an element $\eta$ satisfying $\eta^2=\eta\star\eta=\eta$.
We provide some small examples below.

\begin{example}\label{example:AutHamming}
The Norton algebras $V_1(H(n,e))$ and $V_n(H(n,3))$ are isomorphic to direct products of copies of $V_1(H(1,e))$
by Corollary~\ref{cor:V1Hn} and Proposition~\ref{prop:VnHn}.
We have a basis $\{\chi_1, \ldots, \chi_{e-1}\}$ for $V_1(H(1,e))$.
For $e=2$, the automorphism group of $V_1(H(1,2))$ is given by Proposition~\ref{prop:AutV0Hn}.
For $e=3$, one can check that the nonzero idempotents of the Norton algebra $V_1(H(1,3))$ are $\chi_1+\chi_2$, $\omega\chi_1+\omega^2\chi_2$, and $\omega^2\chi_1+\omega\chi_2$. 
An automorphism of $V_1(H(1,3))$ must permute the three nonzero idempotents.
Combining this with Theorem~\ref{thm:AutHamming} one sees that that the automorphism group of $V_1(H(1,3))$ is isomorphic $\SS_3\cong \ZZ_3\rtimes\ZZ_3^\times$.
This does not hold for $e=4$ though, as the Norton algebra $V_1(H(1,4))$ has four nonzero idempotents $\frac12 \chi_2 \pm \frac12(\chi_1+\chi_3)$ and $-\frac12 \chi_2 \pm \frac i2(\chi_1-\chi_3)$, but the symmetric group $\SS_4$ contains $\ZZ_4\rtimes\ZZ_4^\times$ as a proper subgroup. 
\end{example}

To generalize Example~\ref{example:AutHamming}, we need to examine idempotents carefully.
Let $\phi$ be any automorphism of an algebra $(V,\star)$.
If two idempotents $\eta_1$ and $\eta_2$ in this algebra are \emph{orthogonal}, meaning that $\eta_1\star\eta_2=0$, then so are $\phi(\eta_1)$ and $\phi(\eta_2)$.
If an idempotent $\eta\in V$ is \emph{primitive} in the sense that it cannot be written as the sum of two nonzero orthogonal idempotents, then $\phi(\eta)$ must be a primitive idempotent as well, since otherwise there would exist two nonzero orthogonal idempotents $\eta_1$ and $\eta_2$ such that $\phi(\eta)=\eta_1+\eta_2$, leading to a contradiction that $\eta=\phi^{-1}(\eta_1)+\phi^{-1}(\eta_2)$ with $\phi^{-1}(\eta_1)\star\phi^{-1}(\eta_2) = \phi^{-1}(\eta_1\star\eta_2)=0$.
We will find all idempotents and classify the primitive ones in the Norton algebra $V_1(H(1,e))$ in order to determine its automorphism group.

\begin{proposition}\label{prop:IdempotentsV1}
For $e\ge3$, the Norton algebra $V_1(H(1,e))$ is spanned by distinct nonzero idempotents 
\begin{equation}\label{eq:idempotents}
\eta_j:= \frac{1}{e-2}(\omega^j \chi_1+\omega^{2j}\chi_2+\cdots+\omega^{(e-1)j}\chi_{e-1}) \quad \text{for }j=0,1,\ldots,e-1
\end{equation}
which satisfy $\eta_j\star\eta_k=-\frac{\eta_j+\eta_k}{e-2}$ whenever $j\ne k$ and $\eta_0+\cdots+\eta_{e-1}=0$, 
and any nonzero idempotent in $V_1(H(1,e))$ can be written uniquely as $\frac{e-2}{e-2\ell}(\eta_{j_1}+\cdots+\eta_{j_\ell})$ for some distinct $j_1,\ldots,j_\ell \in[e-1]$. 
\end{proposition}

\begin{proof}
For any $j,k,\ell\in\ZZ_e$ with $\ell\ne0$, the coefficient of $\chi_\ell$ in $\eta_j\star\eta_k$ is
\[ \begin{aligned}
& (\omega^{j}\omega^{(\ell-1)k}+\omega^{2j}\omega^{(\ell-2)k}+\cdots+ \omega^{(e-1)j}\omega^{(\ell+1)k} - \omega^{\ell j}\omega^{(\ell-\ell)k})/(e-2)^2 \\
=& ((\omega^{j-k}+\omega^{2(j-k)}+\cdots+\omega^{(e-1)(j-k)}) \omega^{\ell k} - \omega^{\ell j})/(e-2)^2 \\
=&
\begin{cases}
((e-1)\omega^{\ell j} - \omega^{\ell j})/(e-2)^2 = \omega^{\ell j}/(e-2) & \text{if } j=k \\
- (\omega^{\ell k} + \omega^{\ell j})/(e-2)^2 & \text{if } j\ne k
\end{cases}
\end{aligned} \]
where the second equality follows from Equation~\eqref{eq:root}.
Thus $\eta_j$ is an idempotent for all $j\in \ZZ_e$ and $\eta_j\star\eta_k = -(\eta_j+\eta_k)/(e-2)$ whenever $j\ne k$.
Moreover, the coefficients of $\chi_1,\ldots,\chi_{e-1}$ in $\eta_0,\ldots,\eta_{e-1}$ form a matrix $[ \omega^{jk}]_{0\le j\le e-1,\ 1\le k\le e-1}$ which is a full rank $e\times(e-1)$-submatrix of the $e\times e$ Vandermonde matrix $[ \omega^{jk}]_{j,k=0}^e$.
Thus $\eta_0,\ldots,\eta_{e-1}$ span $V_1(H(1,e))$ and their sum is zero thanks to Equation~\eqref{eq:root}.

Now deleting any element, say $\eta_0$, from the spanning set $\{\eta_0,\ldots,\eta_{e-1}\}$ gives a basis for the Norton algebra $V_1(H(1,e))$ since its dimension is $e-1$.
Thus any element in $V_1(H(1,e))$ can be written uniquely as $c_1\eta_1+\cdots+c_{e-1} \eta_{e-1}$ for some $c_1,\ldots,c_{e-1}\in\CC$.
This element is an idempotent if and only if
\[ \sum_{j=1}^{e-1} c_j^2 \eta_j + \sum_{1\le j\ne k \le e-1} -\frac{c_jc_k(\eta_j+\eta_k)}{e-2} = \sum_{j=1}^{e-1} c_j \eta_j. \]
For $j=1,\ldots,e-1$, taking the coefficient of $\eta_j$ in the above equation gives
\begin{equation}\label{eq:IdempotentV1}
c_j^2 - \sum_{k\ne j} \frac{2c_j c_k}{e-2} = c_j. 
\end{equation}
Suppose that there exist distinct $j,j'\in[e-1]$ such that $c_j\ne0$ and $c_{j'}\ne0$. Then we have
\[ c_j-\sum_{k\ne j} \frac{2c_k}{e-2} = 1 = c_{j'}-\sum_{k\ne j'} \frac{2c_k}{e-2}. \]
This implies  $c_j-c_{j'} + 2(c_j-c_{j'})/(e-2)=0$, that is, $c_j=c_{j'}$.
Hence any nonzero idempotent can be written uniquely as $c(\eta_{j_1}+\cdots+ \eta_{j_\ell})$ for some distinct $j_1,\ldots,j_\ell\in[e-1]$ and some nonzero $c\in\CC$.
Using Equation~\eqref{eq:IdempotentV1} we obtain $c-2(\ell-1)c/(e-2)=1$, so $c=(e-2)/(e-2\ell)$.
\end{proof}

\begin{remark}\label{rem:nilpotents}
In our previous work~\cite{Huang20} we found a real basis $\{\bar v_1,\ldots,\bar v_e\}$ for the Norton algebra $V_1(H(1,e))$; see also Equation~\eqref{eq:V1}.
By Equation~\eqref{eq:root} and Equation~\eqref{eq:idempotents}, for all $j,k\in\ZZ_e$ we have
\[ \eta_j(k) = \frac{1}{e-2}(\omega^j\omega^k + \omega^{2j} \omega^{2k} + \cdots + \omega^{(e-1)j}\omega^{(e-1)k})
= \begin{cases}
\frac{-1}{e-2} & \text{if } j+k\ne0 \\[1ex]
\frac{e-1}{e-2} & \text{if } j+k=0.
\end{cases} \]
Comparing this with our previous work~\cite{Huang20} we have $\eta_j=\bar v_{e-j}\in\RR^X$ for all $j\in\ZZ_e$. 
Furthermore, a similar argument as in the proof of Proposition~\ref{prop:IdempotentsV1} shows that $\eta\in V_1(H(1,e))$ is a \emph{nilpotent element of order $2$}, i.e., $\eta\ne0$ and $\eta^2=0$, if and only if $e=c(\eta_{j_1}+\cdots + \eta_{j_\ell})$ with $\ell=e/2$ for some nonzero scalar $c\in\CC$.
\end{remark}

By Proposition~\ref{prop:IdempotentsV1}, any nonzero idempotent $\eta=\frac{e-2}{e-2\ell}(\eta_{j_1}+\cdots+\eta_{j_\ell})$ in $V_1(H(1,e))$ is determined by its \emph{support} $\supp(\eta):=\{j_1,\ldots,j_\ell\}$ as the coefficient of $\eta_j$ in $\eta$ is the nonzero constant $\frac{e-2}{e-2\ell}$ for all $j\in\supp(\eta)$ and zero for all $j\notin\supp(\eta)$.
In particular, if $|\supp(\eta)|=e-1$ then we have $\eta=-\eta_1-\cdots-\eta_{e-1}=\eta_0$.
We want to show that any automorphism of the Norton algebra $V_1(H(1,e))$ must permute $\eta_0, \eta_1, \ldots, \eta_{e-1}$.
To this end, we study the primitive and orthogonal idempotents in this algebra.

\begin{proposition}\label{prop:PrimitivesV1}
In the Norton algebra $V_1(H(1,e))$ with $e\ge3$, the nonzero idempotents are pairwise nonorthogonal and all primitive, and any  two distinct nonzero idempotents $\eta, \eta'$ give rise to a nonzero idempotent $c(\eta+\eta')$ if and only if
\begin{enumerate}
\item[(i)]
$\supp(\eta)\cap \supp(\eta')=\emptyset$, $\ell=\ell'$, and $c=\frac{e-2\ell}{e-4\ell}$, or
\item[(ii)]
$\supp(\eta)\supseteq\supp(\eta')$, $\ell'=e-\ell$, and $c=\frac{e-2\ell}{3e-4\ell}$.
\end{enumerate}
\end{proposition}

\begin{proof}
Let $\eta$ and $\eta'$ be distinct nonzero idempotents in $V_1(H(1,e))$ with $|\supp(\eta)|=\ell$ and $|\supp(\eta')|=\ell'$.
We may assume that there exists $k\in \supp(\eta)-\supp(\eta')$, without loss of generality.
By Proposition~\ref{prop:IdempotentsV1}, the coefficient of $\eta_k$ in $\eta\star\eta'$ is
\[ -\frac{\ell'}{e-2}\frac{e-2}{e-2\ell}\frac{e-2}{e-2\ell'} = \frac{-\ell'(e-2)}{(e-2\ell)(e-2\ell')}\ne0. \]
This shows that $\eta$ and $\eta'$ are nonorthogonal.
Thus every nonzero idempotent is primitive since it cannot be written as the sum of two orthogonal nonzero idempotents.

Suppose that $c(\eta+\eta')$ is also a nonzero idempotent for some $c\in\CC$.
For each $j\in[e-1]$, the coefficient of $\eta_j$ in $\eta+\eta'$ is
\[ \begin{cases}
\frac{e-2}{e-2\ell} & \text{if } j\in \supp(\eta)-\supp(\eta') \\[1ex]
\frac{e-2}{e-2\ell'} & \text{if } j\in \supp(\eta')-\supp(\eta) \\[1ex]
\frac{e-2}{e-2\ell} + \frac{e-2}{e-2\ell'} & \text{if } j\in \supp(\eta)\cap\supp(\eta')
\end{cases} \]
The first case occurs by our hypothesis on $\eta$ and $\eta'$.
Assume that the second case also occurs.
Since the coefficients in these two cases are nonzero, they must coincide, i.e., $\ell=\ell'$ by Proposition~\ref{prop:IdempotentsV1}.
Then $\supp(\eta)\cap \supp(\eta')=\emptyset$ since the coefficient in the third case is nonzero but distinct from the coefficient in the first two cases.
It follows that the support of the idempotent $c(\eta+\eta')$ has cardinality $4\ell$ and Proposition~\ref{prop:IdempotentsV1} implies $\frac{c(e-2)}{e-2\ell}=\frac{e-2}{e-4\ell}$, i.e., $c=\frac{e-2\ell}{e-4\ell}$.

Now assume that the second case does not occur.
Then the third case must occur since $\supp(\eta')\ne\emptyset$.
As the coefficient in the third case differs from the first, we have $\frac{e-2}{e-2\ell} + \frac{e-2}{e-2\ell'}=0$, i.e., $\ell'=e-\ell$ by Proposition~\ref{prop:IdempotentsV1}.
Then the support of $c(\eta+\eta')$ has $\ell-\ell'=2\ell-e$ elements and thus $\frac{c(e-2)}{e-2\ell} = \frac{e-2}{e-2(2\ell-e)}$, i.e., $c=\frac{e-2\ell}{3e-4\ell}$.
\end{proof}

\begin{proposition}\label{prop:AutV1H1}
For $e\ge3$, the automorphism group of the Norton algebra $V_1(H(1,e))$ is isomorphic to the symmetric group $\SS_e$.
\end{proposition}

\begin{proof}
By Proposition~\ref{prop:IdempotentsV1}, any permutation of the idempotents $\eta_0,\ldots,\eta_{e-1}$ gives an automorphism of the Norton algebra $V_1(H(1,e))$.
It remains to show that any automorphism $\phi$ of the Norton algebra $V_1(H(1,e))$ permutes $\eta_0,\ldots,\eta_{e-1}$.

For any distinct $j,j'\in[e-1]$, we have two nonzero idempotents $\eta:=\phi(\eta_j)$ and $\eta':=\phi(\eta_{j'})$ with $|\supp(\eta)|=\ell$ and $|\supp(\eta')|=\ell'$.
First assume $e\ne 4$, so $c(\eta_j+\eta_{j'})$ is a nonzero idempotent, where $c:=\frac{e-2}{e-4}$.
Then $\phi$ takes it to $c(\eta+\eta')$, which is also a nonzero idempotent.
By Proposition~\ref{prop:PrimitivesV1}, we have $\ell=\ell'=1$ or $\{\ell,\ell'\}=\{1,e-1\}$.
Note that an idempotent with support of cardinality $e-1$ is $-(\eta_1+\cdots+\eta_{e-1})=\eta_0$ by Proposition~\ref{prop:IdempotentsV1}.
Thus $\eta=\eta_k$ and $\eta'=\eta_{k'}$ for some distinct $k,k'\in\ZZ_e$.
This implies that $\phi$ permutes $\eta_0,\ldots,\eta_{e-1}$.

Now assume $e=4$.
By Remark~\ref{rem:nilpotents}, the nilpotent elements of order $2$ in $V_1(H(1,4))$ are all nonzero scalar multiple of $\eta_j+\eta_{j'}$ for distinct $j,j'\in\{1,2,3\}$. 
Then $\eta+\eta'$ is also a nilpotent element of order $2$.
Similarly to Proposition~\ref{prop:PrimitivesV1}, we can show that $|\supp(\eta)|=|\supp(\eta')|=1$ or $\{|\supp(\eta)|, |\supp(\eta')|\}=\{1,3\}$.
Thus $\phi$ permutes $\eta_0,\eta_1,\eta_2,\eta_3$.
\end{proof}

\begin{remark}
Given positive integers $n$ and $k$ with $n\ge 2k$, the Johnson graph $J(n,k)$ is a distance regular graph of diameter $n$ with vertex set $X$ consisting of all $k$-subsets of $[n]$ and edge set $E=\{xy: x,y\in X,\ |x\cap y|=1\}$.
By Maldonado and Penazzi~\cite{NortonAlgebra} (see also our work~\cite{Huang20}), the Norton algebra $V_1(J(n,k))$ has a basis $\{\bar v_1,\ldots,\bar v_{n-1}\}$ satisfying $\bar v_i\star\bar v_j=0$ for all $i,j\in[n-1]$ when $n=2k$ or $\bar v_i\star\bar v_i=\bar v_i$ and $\bar v_i\star\bar v_j = -(\bar v_i+\bar v_j)/(n-2)$ for all distinct $i,j\in[n-1]$ when $n>2k$.
Thus the automorphism group of the algebra $V_1(J(n,k))$ is the general linear group $\mathrm{GL}_{n-1}$ when $n=2k$ or the symmetric group $\SS_n$ since $V_1(J(n,k))\cong V_1(H(1,n))$ when $n>2k$.
Also note that $H(1,n)$ is a complete graph with $n$ vertices, whose automorphism group is $\SS_n$.
\end{remark}

\begin{theorem}\label{thm:AutV1}
The automorphism group of the Norton algebra $V_i(H(n,e))$ is isomorphic to $\SS_e\wr\SS_n$ if $i=1$ and $e\ge3$, or $\SS_3\wr\SS_{2^{n-1}}$ if $i=n$ and $e=3$.
\end{theorem}

\begin{proof}
By Corollary~\ref{cor:Hamming}, the Norton algebra $V_1(H(n,e))$ is isomorphic to the direct product of its subalgebras $V_S$ for all 1-sets $S=\{k\}\subseteq[n]$, i.e., $V_1(H(n,e))\cong V_{\{1\}}\times\cdots\times V_{\{n\}}$.
Thus any idempotent in $V_1(H(n,e))$ can be expressed as $\eta=\eta^{(1)}+\cdots+\eta^{(n)}$ with $\eta^{(j)}$ being an idempotent in $V_{\{j\}}$ for all $j\in[n]$.
By Proposition~\ref{prop:PrimitivesV1}, the idempotent $\eta$ is primitive if and only if all but one of $\eta^{(1)}, \ldots,\eta^{(n)}$ are zero.
Moreover, two primitive idempotents in $V_1(H(n,e))$ are orthogonal if and only if they are from $V_{\{j\}}$ and $V_{\{k\}}$ for some distinct $j,k\in[n]$ by Proposition~\ref{prop:PrimitivesV1}.
Thus any automorphism $\phi$ of the Norton algebra $V_1(H(n,e))$ must permute the subalgebras $V_{\{1\}}, \ldots, V_{\{n\}}$. 
It follows from this and Proposition~\ref{prop:AutV1H1} that the automorphism group of $V_1(H(n,e))$ is the wreath product $\SS_e\wr\SS_n$.
The result on $V_n(H(n,3))$ is similar, thanks to Proposition~\ref{prop:VnHn}.
\end{proof}

\begin{remark}\label{rmk:AutV2H3}
The above theorem shows that the automorphism group of a Norton algebra may agree with the automorphism group of the underlying graph in some cases but could be much larger in some other cases.
It is also possible to have the former smaller than the latter. 
In fact, one can check that the nonzero idempotents in the Norton algebra $V_2(Q_3)$ (see Example~\ref{example:H32}) are $\chi_{12}+\chi_{13}+\chi_{23}$, $\chi_{12}-\chi_{13}-\chi_{23}$, $-\chi_{12}+\chi_{13}-\chi_{23}$, and $-\chi_{12}-\chi_{13}+\chi_{23}$.
Thus its automorphism group is a subgroup of $\SS_4$, and it actually equals $\SS_4$ since it has a subgroup isomorphic to $\SS_3^B/\{\pm1\}$ by Corollary~\ref{cor:AutH2}, whose order is $2^3\cdot3!/2=24=|\SS_4|$.
This is smaller than the automorphism group $\SS_2\wr\SS_3$ of the graph $Q_3$, whose order is $2^3\cdot3!=48$.
It would be nice to generalize this example to the Norton algebra $V_i(H(n,e))$ with $i\ge2$. 
\end{remark}

\section{Nonassociativity}\label{sec:NonAssoc}

In this section we study the nonassociativity of the Norton product $\star$ on each eigenspace $V_i$ the Hamming graph $H(n,e)$ based on 
the results in Section~\ref{sec:Hamming}.

Given a binary operation $*$ defined on a set $Z$, define $C_{*,m}$ to be the number of distinct results that one can obtain from the expression $z_0*z_1*\cdots*z_m$ by inserting parentheses in all possible ways, where $z_0, z_1, \ldots,z_m$ are indeterminates taking values from $Z$.
We have $C_{*,m}\ge1$ and the equality holds for all $m\ge0$ if and only if $*$ is associative.
On the other hand, it is well known that the number of ways to insert parentheses into the expression $z_0*z_1*\cdots*z_m$ is the ubiquitous \emph{Catalan number} $C_m:=\frac{1}{m+1}\binom{2m}{m}$, giving an upper bound for $C_{*,m}$.
If $C_{*,m}=C_m$ for all $m\ge0$ then $*$ is said to be \emph{totally nonassociative}.
In general, the number $C_{*,m}$ is between $1$ and $C_m$ and can be viewed as a quantitative measure for how far the operation $*$ is from being associative or totally nonassociative~\cite{AssociativeSpectra1, CatMod}

There is a natural bijection between the ways to insert parentheses into $z_0*z_1*\cdots*z_m$ and binary trees with $m+1$ leaves.
Here a \emph{binary tree} is a rooted plane tree in which every node has exactly two children except the leaves.
Let $\T_m$ be the set of all binary trees with $m+1$ leaves.
Any $t\in T_m$ naturally corresponds to a parenthesization of $z_0*z_1*\cdots*z_m$, and we let $(z_0*z_1*\cdots*z_m)_t$ to denote the result. 
With the leaves of $t$ labeled $0,1,\ldots,m$ from left to right, the \emph{depth sequence} of $t$ is $d(t):=(d_0(t), d_1(t), \ldots, d_m(t))$, where the \emph{depth} $d_j(t)$ of a leaf $j$ in $T$ is the number of steps in the unique path from the root of $t$ to the leaf $j$.
See Figure~\ref{fig:parenthesization} for some examples.

\begin{figure}[h]
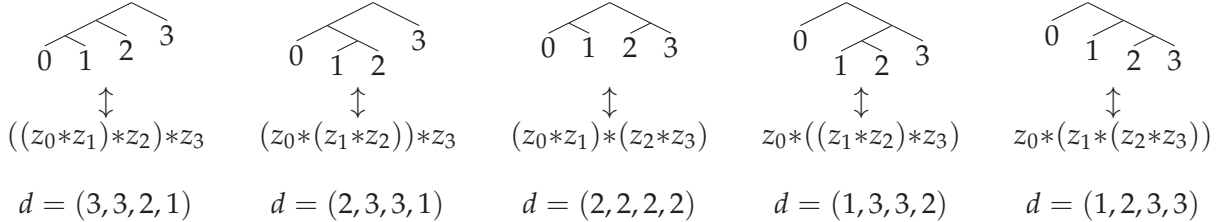

\[\begin{array}{ccccc}
\Tree [.  [. [. 0 1 ] 2 ] 3 ] &
\Tree [.  [. 0 [. 1 2 ] ] 3 ] &
\Tree [.  [. 0 1 ] [. 2 3 ] ] &
\Tree [. 0 [. [. 1 2 ] 3 ] ] &
\Tree [. 0  [. 1 [. 2 3 ] ] ] \\
\updownarrow & \updownarrow & \updownarrow & \updownarrow & \updownarrow \\
\rule{5pt}{0pt}((z_0 { * } z_1) { * } z_2) { * } z_3\rule{5pt}{0pt} &
\rule{5pt}{0pt}(z_0 { * } (z_1{ * } z_2)) { * } z_3 \rule{5pt}{0pt} &
\rule{5pt}{0pt}(z_0 { * } z_1) { * } (z_2 { * } z_3)\rule{5pt}{0pt} &
\rule{5pt}{0pt}z_0 { * } ((z_1 { * } z_2) { * } z_3)\rule{5pt}{0pt} &
\rule{5pt}{0pt}z_0 { * } (z_1 { * } (z_2 { * } z_3))\rule{5pt}{0pt} \\ \\
d = (3,3,2,1) & d = (2,3,3,1) & d = (2,2,2,2) & d = (1,3,3,2) & d = (1,2,3,3) \\
\end{array} \]
\caption{Parenthesizations, binary trees, and leaf depths}\label{fig:parenthesization}
\end{figure}

Two binary trees $s,t\in\T_m$ are \emph{$*$-equvalent} if $(z_0*z_1*\cdots*z_m)_s = (z_0*z_1*\cdots*z_m)_t$.
The number of equivalence classes in $\T_m$ is exactly the nonassociativity measure $C_{*,m}$.
Two binary operations are said to be \emph{equally nonassociative} if their corresponding equivalence relations agree on $\T_m$ for all $m\ge0$.

Previously, the author, Mickey, and Xu~\cite{DoubleMinus} defined the \emph{double minus operation} $a\ominus b:= -a-b$ for all $a,b\in \CC$ and discovered that the nonassociativity measure $C_{\ominus,m}$ agrees with an interesting sequence A000975~\cite{A000975} in OEIS~\cite{OEIS}; see also Cs\'{a}k\'{a}ny and Waldhauser~\cite{AssociativeSpectra1}.

\begin{theorem}[\cite{DoubleMinus}]\label{thm:ominus}
For any $s,t\in\T_m$, we have $(z_0\ominus z_1\ominus \cdots \ominus z_m)_s = (z_0\ominus z_1\ominus \cdots \ominus z_m)_t$ if and only if $d(s)\equiv d(t) \pmod 2$, i.e., $d_j(s) \equiv d_j(t) \pmod 2$ for all $j=0,1,\ldots,m$.
Moreover, the nonassociativity measure $C_{\ominus, m}$ is given by the the sequence A000975 in OEIS except for $m=0$. 
\end{theorem}

Now we turn to the Norton algebras of the Hamming graph $H(n,e)$, which are related to the double minus operation in certain cases.
The Norton algebra $V_0(H(n,e))$ is one dimensional and must be associative. 
In recent work~\cite{Huang20} we obtained the following result on the nonassociativity of the Norton algebra $V_1(H(n,e))$.

\begin{proposition}[\cite{Huang20}]\label{prop:V1}
The Norton product $\star$ on $V_1(H(n,e))$ is associative if $e=2$, equally as nonassociative as the double minus operation $\ominus$ if $e=3$, and totally nonassociative if $e\ge4$.
\end{proposition}

We extend this result to the Norton algebras $V_i(H(n,e))$ for $i\ge2$ in the next few subsections.

\subsection{The case $e=3$}
Assume $e=3$ in this subsection.
We begin with the subcase $i=n$.

\begin{proposition}
The Norton product $\star$ on $V_n(H(n,3))$ is equally as nonassociative as the double minus operation $\ominus$.
\end{proposition}

\begin{proof}
By Proposition~\ref{prop:VnHn}, the Norton algebra $V_n(H(n,3))$ is isomorphic to the direct product of $2^{n-1}$ copies of $V_1(H(1,3))$.
Thus the result holds by Proposition~\ref{prop:V1}.
\end{proof}

Now assume $1<i<n$. 
We need to recall some terminology for binary trees.
Given a vertex $x$ in a binary tree $t$, the \emph{(maximal) subtree of $t$ rooted at $x$} consists of all vertices and edges weakly below $x$.
In particular, the \emph{left/right subtree of $t$} is the subtree rooted at the left/right child of the root of $t$.

\begin{lemma}\label{lem:H3}
For any $m\ge0$, $t\in \T_m$, $j\in [m]$, and $u\in X_i$, there exist $z_0, z_1, \ldots, z_m\in \{\chi_u, \chi_{2u}\}$ such that $z_j=\chi_u$ and that $(z_0\star z_1\star\cdots \star z_m)_t$ equals $\chi_u$ if $d_j(t)$ is even or $\chi_{2u}$ if $d_j(t)$ is odd.
\end{lemma}

\begin{proof}
We induct on $m$.
For $m=0$, the result is trivial.
Assume $m\ge1$ below.

Let $t_1\in \T_{m_1}$ and $t_2\in \T_{m_2}$ be the left and right subtrees of $t$.
Then
\[ 
(z_0\star \cdots \star z_m)_t = (z_0\star\cdots\star z_{m_1})_{t_1} \star (z_{m_1+1}\star\cdots\star z_m)_{t_2}.
\]
Assume $j\in[m_1]$, without loss of generality.
Consider the case when $d_j(t)$ is even; the odd case is similar.
We have $d_j(t_1)$ is odd as $d_j(t)=d_j(t_1)+1$.
We can obtain $z_0,\ldots,z_m\in \{\chi_u,\chi_{2u}\}$ such that $(z_0\star\cdots\star z_{m_1})_{t_1} =(z_{m_1+1}\star\cdots\star z_m)_{t_2}=\chi_{2u}$ by applying the inductive hypothesis to $t_1$ (with $z_j=\chi_u$) and $t_2$ (with $z_m=\chi_u$ if $d_m(t_2)$ is odd or $z_m=\chi_{2u}$ if $d_m(t_2)$ is even).
Then we have $(z_0\star \cdots \star z_m)_t = \chi_{2u}\star\chi_{2u}=\chi_u$ as desired.
\end{proof}

\begin{theorem}\label{thm:NonAssocH3}
For $i=2,3,\ldots,n-1$, the Norton product on $V_i(H(n,3))$ is totally nonassociative.
\end{theorem}

\begin{proof}
Let $s$ and $t$ be any two distinct binary trees in $\T_m$.
We need to show that
\begin{equation}\label{eq:unequal}
(z_0\star \cdots \star z_m)_s \ne (z_0\star \cdots \star z_m)_t \quad\text{ for some $z_0, z_1, \ldots, z_m\in V_i(H(n,3))$.}
\end{equation}
We may assume that $d_s(j)\equiv d_j(t)\pmod 2$ for all $j=0,1,\ldots,m$; if not, then there exist $z_0, z_1, \ldots, z_m$ in the subalgebra $V_u\cong V_1(H(1,3))$ of $V_i(H(n,3))$ (see Lemma~\ref{lem:Vu}) for any $u\in X_i$ such that $(z_0\star \cdots \star z_m)_s \ne (z_0\star \cdots \star z_m)_t$ by Theorem~\ref{thm:ominus} and Proposition~\ref{prop:V1}. 

We proceed by induction on $m$. 
For $m=2$, Equation~\eqref{eq:unequal} holds since any $u\in X_i$ satisfies
\[ (\chi_u\star\chi_u)\star \chi_{-u} = \chi_{2u}\star\chi_{-u} =\chi_u \ne 0 = \chi_u\star(\chi_u\star \chi_{-u}). \]

Now assume $m\ge3$. 
Let $j$ be the leftmost leaf with the largest depth among all leaves in $s$.
Then $j$ is a left leaf, $j+1$ is a right leaf, and they share a common parent in $s$.
We distinguish some cases below for the positions of $j$ and $j+1$ in $t$.

\vskip5pt\noindent\textbf{Case 1}.
Suppose that $j$ is a left leaf and $j+1$ is a right leaf, so they share a common parent in $t$.
Then deleting $j$ and $j+1$ from $s$ and $t$ gives two distinct trees $s'$ and $t'$ in $\T_{m-1}$.
Applying the inductive hypothesis to $s'$ and $t'$ gives $(z'_0\star\cdots\star z'_{m-1})_{s'} \ne (z'_0\star\cdots\star z'_{m-1})_{t'}$ for some $z'_0, \ldots, z'_{m-1}$ in $\{\chi_u: u\in X_i\}$.
We have $z'_j=\chi_v$ for some $v\in X_i$ and we can define $z_j=z_{j+1}:=\chi_{2v}$.
Also let $z_k:=z'_k$ for $k=0,\ldots,j-1$ and $z_\ell:=z'_{\ell-1}$ for $\ell=j+2,\ldots,m$.
Since $z_j\star z_{j+1} = \chi_v = z'_j$, we have 
\[ (z_0\star\cdots\star z_m)_s = (z'_0\star\cdots\star z'_{m-1})_{s'} 
\ne (z'_0\star\cdots\star z'_{m-1})_{t'} = (z_0\star\cdots\star z_m)_t. \]

\vskip5pt\noindent\textbf{Case 2}.
Suppose that $j$ and $j+1$ are both left leaves in $t$.
Then $j+1$ is contained in the subtree $r$ of $t$ rooted at the right sibling of $j$.
Since $d_j(t)\equiv d_j(s)=d_{j+1}(s)\equiv d_{j+1}(t) \pmod 2$, the depth of $j+1$ in $r$ must be even.
Thus the left subtree of $r$ has two left and right subtrees $r_1$ and $r_2$ and $j+1$ is in $r_1$ with an even depth.
Define $u,v,w\in X_i$ below such that $u+v\notin X_i$, $u+w\in X_i$, and $v+w\in X_i$.
\begin{center}
\begin{tabular}{c|ccccccccc}
$j$ & $1$ & $2$ & $3$ & $\cdots$ & $i$ & $i+1$ & $\cdots$ &  $n-1$ & $n$ \\ 
\hline
$u(j)$ & $1$ & $1$ & $1$ & $\cdots$ & $1$ & $0$ & $\cdots$ & $0$ & $0$ \\
$v(j)$ & $0$ & $1$ & $1$ & $\cdots$ & $1$ & $0$ & $\cdots$ & $0$ & $1$ \\
$w(j)$ & $2$ & $0$ & $1$ & $\cdots$ & $1$ & $0$ & $\cdots$ & $0$ & $2$  
\end{tabular}
\end{center}
By Lemma~\ref{lem:H3}, the subtree $r_1$ can produce $\chi_{v}$ with $z_{j+1}=\chi_v$, the subtree $r_2$ can produce $\chi_w$, and the right subtree of $r$ can product $\chi_{2v}$.
Combining these with $z_j=\chi_u$ gives 
\[ \chi_u \star(( \chi_v\star\chi_w)\star\chi_{2v}) = \chi_u\star(\chi_{v+w}\star \chi_{2v}) 
= \chi_u\star\chi_w =\chi_{u+w}.\]
See the left picture in Figure~\ref{fig:tree}, where $j$ and $j+1$ are in red.
Then applying Lemma~\ref{lem:H3} to the tree obtained from $t$ by contracting $j$ and $r$ to their parent gives $(z_0\star\cdots\star z_m)_t = \chi_{c(u+w)}\ne0$, where $c\in\{1,2\}$.
On the other hand, we have $(z_0\star\cdots\star z_m)_s=0$ since $z_j\star z_{j+1}=\chi_u\star\chi_v=0$. 
Thus we are done with this case.

\begin{figure}[h]
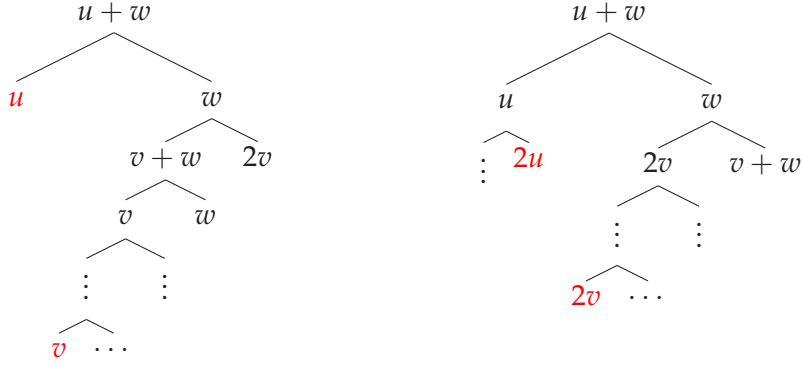

\Tree [.$u+w$ [.${\color{red}u}$ ] [.$w$ [.$v+w$ [.$v$ [.$\vdots$ ${\color{red}v}$ $\cdots$ ] $\vdots$ ] $w$ ] $2v$ ] ]
\qquad
\Tree [.$u+w$ [.$u$ $\vdots$ ${\color{red}2u}$ ] [.$w$ [.$2v$ [.$\vdots$ ${\color{red}2v}$ $\cdots$ ] $\vdots$ ] $v+w$ ] ]
\caption{Case 2 and Case 3 in the proof of Theorem~\ref{thm:NonAssocH3}}\label{fig:tree}
\end{figure}

\vskip5pt\noindent\textbf{Case 3}.
Suppose that $j$ is a right leaf in $t$.
Then $j$ is contained in the subtree $r_1$ of $t$ rooted at the parent of $j$, and $j+1$ is contained in the subtree $r_2$ of $t$ rooted at the right sibling of the parent of $j$.
Since $d_j(t)\equiv d_j(s)=d_{j+1}(s)\equiv d_{j+1}(t) \pmod 2$, the depth of $j+1$ in $r_2$ must be odd.
Thus $j+1$ must be a left leaf in the left subtree of $r_2$ with an even depth.
By Lemma~\ref{lem:H3}, we can obtain $\chi_{u}$ from $r_1$ with $z_j=\chi_{2u}$, obtain $\chi_{2v}$ from the left subtree of $r_2$ with $z_{j+1}=\chi_{2v}$, and obtain $\chi_{v+w}$ from the right subtree of $r_2$.
Combining $r_1$ and $r_2$ gives 
\[ \chi_{u}\star(\chi_{2v}\star\chi_{v+w}) = \chi_{u}\star\chi_{w} = \chi_{u+w}. \] 
See the right picture in Figure~\ref{fig:tree}, where $j$ and $j+1$ are in red.
Contracting $r_1$ and $r_2$ and applying Lemma~\ref{lem:H3} again gives $(z_0\star\cdots\star z_m)_t = \chi_{c(u+w)}\ne0$, where $c\in\{1,2\}$.
On the other hand, we have $(z_0\star\cdots\star z_m)_s=0$ since $z_j\star z_{j+1}=\chi_{2u}\star\chi_{2v}=0$.
\end{proof}

\subsection{The case $e\ge4$}
We study case $e\ge4$ similarly as the case $e=3$.
We may assume $i\ge1$. 

\begin{lemma}\label{lem:H4}
For any $m\ge1$, $t\in \T_m$, $j\in [m]$, $u\in X_i$, and $c\in\{2,\ldots,e-1\}$, there exist $z_0, z_1, \ldots, z_m$ in $\{\chi_v: v\in X_i\}$ such that $z_j=\chi_u$ and $(z_0\star z_1\star\cdots \star z_m)_t = \chi_{cu}$.
\end{lemma}

\begin{proof}
We induct on $m$.
For $m=1$, we have $(c-1)u\in X_i$ and $\chi_u\star\chi_{(c-1)u}=\chi_{(c-1)u}\star\chi_u=\chi_{cu}$.
Assume $m\ge2$ below.
Let $t_1\in \T_{m_1}$ and $t_2\in \T_{m_2}$ be the left and right subtrees of $t$.
Assume $j\in[m_1]$, without loss of generality.
By the inductive hypothesis, for all $c_1=2,\ldots,e-1$, there exist $z_0,\ldots,z_{m_1}\in \{\chi_v:v\in X_i\}$ such that $z_j=\chi_u$ and $(z_0\star\cdots\star z_{m_1})_{t_1}=\chi_{c_1u}$.

If $m_2=0$ then we have $(z_0\star \cdots \star z_m)_t=\chi_{c_1u}\star z_m = \chi_{cu}$, where $c_1=c-1$ and $z_m=\chi_u$ when $c>2$, or $c_1=e-1$ and $z_m=\chi_{3u}$ when $c=2$.

If $m_2\ge1$ then we can apply inductive hypothesis to $t_2$ and get $(z_0\star \cdots \star z_m)_t=\chi_{c_1u}\star \chi_{w} = \chi_{cu}$, where $c_1=c-1$ and $w=(e-1)(-u)=u$ when $c>2$, or $c_1=e-1$ and $w=3u$ when $c=2$.
\end{proof}

\begin{theorem}\label{thm:NonAssocH4}
For $e\ge4$ and $i=1,\ldots,n$, the Norton product on $V_i(H(n,e)))$ is totally nonassociative.
\end{theorem}

\begin{proof} 
The proof is similar to Theorem~\ref{thm:NonAssocH3} but less technical.
Let $s$ and $t$ be two distinct binary trees in $\T_m$.
Define $u\in X_i$ by $u(1)=\cdots=u(i)=1$ and $u(i+1)=\cdots=u(n)=0$.
Then $cu\in X_i$ for all $c=1,\ldots,e-1$.
We show that $(z_0\star \cdots \star z_m)_s \ne (z_0\star \cdots \star z_m)_t$ for some $z_0, z_1, \ldots, z_m$ in $\{\chi_{cu}: c=1,\ldots,e-1\}$ by induction on $m$.
For $m=2$ we have
\[ (\chi_u\star\chi_u)\star \chi_{-u} = \chi_{2u}\star\chi_{-u} =\chi_u \ne 0 = \chi_u\star(\chi_u\star \chi_{-u}). \]

Assume $m\ge3$ below.
Let $j$ be the leftmost leaf of $s$ with the largest depth.
Then $j$ is a left leaf, $j+1$ is a right leaf, and they share a common parent in $s$.
We distinguish some cases below for the positions of $j$ and $j+1$ in $t$.

\vskip5pt\noindent\textbf{Case 1}.
Suppose that $j$ is a left leaf and $j+1$ is a right leaf, so they share a common parent in $t$.
Then deleting $j$ and $j+1$ from $s$ and $t$ gives two distinct trees $s'$ and $t'$ in $\T_{m-1}$.
Applying the inductive hypothesis to $s'$ and $t'$ we obtain the desired result.

\vskip5pt\noindent\textbf{Case 2}.
Suppose that $j$ and $j+1$ are both left leaves in $t$.
Then $j+1$ is contained in the subtree of $t$ rooted at the right sibling of $j$.
By Lemma~\ref{lem:H4}, this subtree can produce $\chi_{2u}$ with $z_{j+1}=\chi_u$.
With $z_j=\chi_{-u}$, the subtree of $t$ rooted at the parent of $j$ gives $\chi_{-u}\star\chi_{2u}=\chi_u$.
Applying Lemma~\ref{lem:H4} again gives $(z_0\star\cdots\star z_m)_t = \chi_{cu}\ne 0$, where $c\in\{2,\ldots,e-1\}$.
On the other hand, we have $(z_0\star\cdots\star z_m)_t = 0$ since $z_j\star z_{j+1} = \chi_{-u}\star\chi_u=0$.
So we are done with this case.

\vskip5pt\noindent\textbf{Case 3}.
Suppose that $j$ is a right leaf in $t$.
Then $j$ is contained in the subtree of $t$ rooted at the parent of $j$, and $j+1$ is contained in the subtree of $t$ rooted at the right sibling of the parent of $j$.
These two subtrees can produce $\chi_{-2u}$ and $\chi_{3u}$, respectively, with $z_j=-u$ and $z_{j+1}=u$, thanks to Lemma~\ref{lem:H4}.
Thus we can make sure $(z_0\star\cdots\star z_m)_t=\chi_{cu}$ for any $c\in\{2,\ldots,e-1\}$, while $(z_0\star\cdots\star z_m)_s = 0$.
This completes the proof.
\end{proof}

\subsection{The case $e=2$}\label{sec:hypercube}

Finally, we study the nonassociativity of the Norton algebras of the hypercube $Q_n=H(n,2)$.

\begin{lemma}\label{lem:tree}
Let $S$ and $T$ be $i$-subsets of $[n]$ with $|S\triangle T|=i$.
Then for any $m\ge0$, $t\in \T_m$, $j\in[m]$, and distinct $R, R'\in \{S, T, S\triangle T \}$, there exist $z_0, \ldots, z_m \in\{\chi_S, \chi_T, \chi_{S\triangle T} \}$ such that $(z_0\star\cdots \star z_m)_t=\chi_R$ and unless $m=0$, this can be done in such a way that $z_j=\chi_{R'}$.
\end{lemma}

\begin{proof}
We induct on $m$.
The result is trivial when $m=0$, and it holds when $m=1$ since $\chi_S \star \chi_T = \chi_{S\triangle T}$, $\chi_S\star \chi_{S\triangle T} = \chi_T$, $\chi_T\star \chi_{S\triangle T} = \chi_S$, and $\star$ is commutative.

Assume $m\ge2$ below.
Let $t_1\in \T_{m_1}$ and $t_2\in \T_{m_2}$ be the left and right subtrees of $t$.
For the same reason as mentioned in the case $m=1$, we may assume $j\in[m_1]$, $R=S$, and $R'=T$, without loss of generality.
If $m_1=0$ then $j=0$ and there exist $z_0,\ldots,z_m\in\{\chi_S,\chi_T,\chi_{S\triangle T}\}$ such that $z_0=\chi_T$, $(z_1\star\cdots\star z_m)_{t_2}=\chi_{S\triangle T}$, and thus $(z_0\star \cdots \star z_m)_t = \chi_T \star \chi_{S\triangle T} = \chi_S$ by the inductive hypothesis.
If $m_1\ge1$ then the inductive hypothesis gives $z_0, \ldots, z_m \in \{\chi_S, \chi_T, \chi_{S\triangle T} \}$ such that $z_j=\chi_T$, $(z_0\star\cdots\star z_{m_1})_{t_1} = \chi_{S\triangle T}$, $(z_{m_1+1}\star\cdots\star z_m)_{t_2}=\chi_T$, and $(z_0\star \cdots \star z_m)_t = \chi_{S\triangle T}\star\chi_T = \chi_S$.
The proof is complete.
\end{proof}

\begin{theorem}\label{thm:NonAssocQn}
The Norton product $\star$ on $V_i(Q_n)$ is associative if $i=0$, $i$ is odd, or $i>\lfloor 2n/3 \rfloor$, and totally nonassociative if $i$ is even and $1\le i\le \lfloor 2n/3 \rfloor$.
\end{theorem}

\begin{proof}
The Norton algebra $V_0(Q_n)$ is one-dimensional and must be associative.
If $i$ is odd or $i>\lfloor 2n/3 \rfloor$ then the algebra $V_i(Q_n)$ is also associative since $\chi_S \star \chi_T = 0$ for all $i$-sets $S, T\subseteq [n]$, thanks to Corollary~\ref{cor:hypercube} and Lemma~\ref{lem:Delta}.

Assume $i$ is even and $1\le i\le \lfloor 2n/3 \rfloor$ below.
There exist $i$-subsets $S$ and $T$ of $[n]$ such that $|S\triangle T|=i$ by Lemma~\ref{lem:Delta}.
Let $s$ and $t$ be two distinct binary trees in $\T_m$.
We show that $(z_0\star \cdots \star z_m)_s \ne (z_0\star \cdots \star z_m)_t$ for some $z_0, z_1, \ldots, z_m\in \{ \chi_S, \chi_T, \chi_{S\Delta T} \}$ by induction on $m$. For $m=2$ we have 
\[ (\chi_S\star\chi_S)\star \chi_T = 0 \ne \chi_T = \chi_S\star(\chi_S\star \chi_T). \]
Assume $m\ge3$. 
Let $j$ be the leftmost leaf of $s$ with the largest depth.
Then $j$ is a left leaf, $j+1$ is a right leaf, and they share a common parent in $s$.
We distinguish some cases for $j$ and $j+1$ in $t$.

\vskip5pt\noindent\textbf{Case 1}.
Suppose that $j$ is a left leaf and $j+1$ is a right leaf, so they share a common parent in $t$.
Then deleting $j$ and $j+1$ from $s$ and $t$ gives two distinct trees $s'$ and $t'$ in $\T_{m-1}$.
Applying the inductive hypothesis gives the desired result.

\vskip5pt\noindent\textbf{Case 2}.
Suppose that $j$ and $j+1$ are both left leaves in $t$.
Then $j+1$ is contained in the subtree of $t$ rooted at the right sibling of $j$.
By Lemma~\ref{lem:tree}, this subtree can produce $\chi_T$ with $z_{j+1}=\chi_S$.
Then the subtree of $t$ rooted at the parent of $j$ gives $\chi_S\star \chi_T = \chi_{S\triangle T}$ with $z_j=\chi_S$.
Hence we can make sure $(z_0\star\cdots\star z_m)_t\in \{\chi_S, \chi_T, \chi_{S\triangle T}\}$ by using Lemma~\ref{lem:tree} again. 
On other hand, we have $(z_0\star\cdots\star z_m)_s=0$ as $z_j\star z_{j+1}=\chi_S\star\chi_S=0$.
So we are done with this case.

\vskip5pt\noindent\textbf{Case 3}.
Suppose that $j$ is a right leaf in $t$.
Then $j$ is contained in the subtree $r_1$ of $t$ rooted at the parent of $j$, and $j+1$ is contained in the subtree $r_2$ of $t$ rooted at the right sibling of the parent of $j$.
With $z_j=z_{j+1}=\chi_S$, Lemma~\ref{lem:tree} implies that $r_1$ can produce $\chi_T$ and $r_2$ can produce either $\chi_S$ if $r_2\in \T_0$ or $\chi_{S\triangle T}$ otherwise.
Combining $r_1$ and $r_2$ gives either $\chi_T\star\chi_S=\chi_{S\triangle T}$ or $\chi_T\star\chi_{S\triangle T}=\chi_S$.
Applying Lemma~\ref{lem:tree} again gives $(z_0\star\cdots\star z_m)_t \in\{ \chi_S,\chi_T, \chi_{S\triangle T}\}$, whereas $(z_0\star\cdots\star z_m)_s=0$.
This completes the proof.
\end{proof}

\begin{remark}
The above proof replies on the fact that $\chi_A \star \chi_B = \chi_C$ for any permutation $\chi_A, \chi_B, \chi_C$ of the triple $\chi_S, \chi_T, \chi_{S\Delta T}$.
Thus one can use the same proof to show that the cross product on the $n$-dimensional space is totally nonassociative for all $n\ge3$, and that  the octonions (hence all higher dimensional Cayley--Dickson constructions) have a totally nonassociative multiplication as well, thanks to the existence of a triple that behaves in the same way as above, except for the anticommutativity which does not affect the proof.
\end{remark}

\section{Halved and/or folded cubes}\label{sec:cubes}

In this section we study the Norton algebras of the halved and/or folded cubes via the same linear character approach used for the hypercube.

\subsection{Halved cube}\label{sec:halved-cube}

Let $\Gamma$ be a distance regular graph of diameter $d$.
For $i=0,1,\ldots,d$, let $\Gamma_i$ be the graph with the same vertex set as $\Gamma$ but with edge set consisting of all unordered pairs of vertices at distance $i$ from each other in $\Gamma$.
If the graph $\Gamma$ is bipartite then $\Gamma_2$ has two connected components, giving two distance regular graphs known as the \emph{halved graphs} of $\Gamma$~\cite[Proposition 2.13]{DistReg2}.

In particular, the \emph{halved cube} or \emph{half-cube} $\frac12 Q_n$ has vertex set $X$ consisting of all binary strings of length $n$ with even weight and edge set $E$ consisting of all unordered pairs of vertices differing in exactly two positions. 
This is a distance regular graph of diameter $d=\lfloor n/2 \rfloor$.
For $i=0,1,\ldots,d$, the $i$th eigenvalue and its multiplicity are~\cite[\S9.2D]{DistReg1}
\[ \theta_i=\frac{(n-2i)^2-n}{2} \qand
\dim(V_i)=
\begin{cases}
\binom{n}{i} & \text{if } i<n/2 \\
\frac12\binom{n}{i} & \text{if } i=n/2.
\end{cases} \]
 
With the vertex set $X$ viewed as a subgroup of $\ZZ_2^n$, the halved cube $\frac12 Q_n$ becomes the Cayley graph $\Gamma(X,X_2)$ of $X$ with respect to $X_2$, where $X_i:=\{x\in X: |\supp(x)|=i\}$.
For every $S\subseteq[n]$ we define a linear character $\chi_S$ of $X$ by 
\[ \chi_S(x) := \prod_{j\in S} (-1)^{x(j)} \quad\text{for all } x\in X. \]

\begin{lemma}\label{lem:halved-cube}
Let $S,T\subseteq[n]$.
Then $\chi_S=\chi_T$ if and only if $S=T$ or $S^c=T$, where $S^c:=[n]-S$.
\end{lemma}

\begin{proof}
For each $x\in X$, since $|\supp(x)|$ is even, we have 
\[ \chi_{S}(x) =  (-1)^{\sum_{j\in S}x(j)} = (-1)^{\sum_{j\in S^c}x(j)} = \chi_{S^c}(x). \]
Thus $\chi_S = \chi_{S^c}$.
Conversely, suppose that $\chi_S=\chi_T$ for two distinct sets $S,T\subseteq[n]$.
Then there exists $j\in S-T$.
If there exists $k\in[n]-S\triangle T$ then we have a contradiction that $\chi_S(x)\ne\chi_T(x)$ for $x\in X$ with $x(j)=x(k)=1$ and all other entries zero.
Thus $S\triangle T=[n]$, which implies $S^c=T$.
\end{proof}

\begin{theorem}\label{thm:halved-cube}
The Norton algebra $V_i(\frac12 Q_n)$ has a basis $\{\chi_S: S\subseteq[n],\ |S|=i\}$ if $0\le i<n/2$ or $\{\chi_S: S\subseteq[n],\ |S|=i,\ 1\in S\}$ if $i=n/2$.
For any $S,T\subseteq[n]$ with $|S|=|T|=i$ we have
\[ \chi_S \star \chi_T =
\begin{cases}
\chi_{S\triangle T} = \chi_{(S\triangle T)^c} & \text{if } |S\triangle T|\in\{i,n-i\} \\
0 & \text{otherwise.} 
\end{cases} \]
\end{theorem}

\begin{proof}
By Lemma~\ref{lem:halved-cube}, the set $\{\chi_S:S\subseteq[n]\}$ has cardinality $2^{n-1}=|X|$ and must contain all linear characters of $X$, thanks to Theorem~\ref{thm:character}.
By Theorem~\ref{thm:Cayley}, if $|S|=i$ then the linear character $\chi_S$ is an eigenvector corresponding the the eigenvalue
\[ \chi_S(X_2) = \binom{i}{2} - i(n-i) + \binom{n-i}{2} = 
\frac{(n-2i)^2-n}2. \]
This proves the desired basis of $V_i(\frac12 Q_n)$ for $i=0,1,\ldots,\lfloor n/2 \rfloor$.
Let $S,T$ be $i$-subsets of $[n]$ with $|S\triangle T|=j$. 
We have $\chi_S\cdot \chi_T = \chi_{S\triangle T} = \chi_{(S\triangle T)^c}$, which belongs to $V_j(\frac12 Q_n)$ if $j\le n/2$ or $V_{n-j}(\frac12 Q_n)$ otherwise. 
Thus the projection onto $V_i$ fixes $\chi_S\cdot \chi_T$ if $j\in\{i,n-i\}$ or annihilates it otherwise. 
\end{proof}

\begin{remark}
There is a bijection from binary strings of length $n$ with even weight to binary strings of length $n-1$ by deleting the $n$th entry.
This gives another way to realize the halved cube $\frac12 Q_n$ as a Cayley graph of a finite abelian group and leads to a slightly different (but equivalent) description of its Norton algebras.
\end{remark}

\begin{corollary}\label{cor:halved-cube}
There is an algebra isomorphism $V_i(\frac12 Q_n)\cong V_i(Q_n)$ if $i< \lceil n/3 \rceil$ or $n-i$ is odd.
\end{corollary}

\begin{proof}
Comparing Theorem~\ref{thm:NortonHamming} with Theorem~\ref{thm:halved-cube}, we see an obvious vector space isomorphism $V_i(\frac12 Q_n)\cong V_i(Q_n)$ for all $i<n/2$, which becomes an algebra isomorphism if $i< \lceil n/3 \rceil$ or $n-i$ is odd, since 
there exist $i$-sets $S,T\subseteq [n]$ such that $|S\triangle T|=n-i$ if and only if $n-i \le 2i$ and $n-i$ is even by Lemma~\ref{lem:Delta}.
\end{proof}

We next examine some examples of the Norton algebra $V_{n/2}(\frac12 Q_n)$, which can be obtained from $V_{n/2}(Q_n)$ by identifying $\chi_S$ and $\chi_{S^c}$ for all $S\subseteq[n]$ with $|S|=n/2$.

\begin{example}\label{example:halved-cube}
The Norton algebra $V_1(\frac12 Q_2)$ has a basis $\{\chi_1=\chi_2\}$ with $\chi_1\star\chi_1=0$, and thus is isomorphic to $V_1(Q_1)$.
The Norton algebra $V_2(\frac12 Q_4)$ has a basis $\{\chi_{12}, \chi_{13}, \chi_{14}\}$ consisting of nilpotent elements of order $2$ satisfying 
\[ \chi_{12}\star\chi_{13} = \chi_{23}=\chi_{14},\quad \chi_{12}\star\chi_{14}=\chi_{24}=\chi_{13}, \qand \chi_{13}\star\chi_{14}=\chi_{34}=\chi_{12}.\]
Comparing this with Example~\ref{example:H32} one sees an algebra isomorphism $V_2(\frac12 Q_4) \cong V_2(Q_3)$.
Incidentally, $V_1(\frac12 Q_3)$ has a basis $\{\chi_1,\chi_2,\chi_3\}$ with $\chi_j\star\chi_j=0$ for $j=1,2,3$ and $\chi_j\star\chi_k=\chi_\ell$ for distinct $j,k,\ell\in\{1,2,3\}$, and thus is isomorphic to $V_2(Q_3)$ as well.
\end{example}



Now we study the automorphism group of the Norton algebra $V_i(\frac12 Q_n)$ using the Coxeter group $\SS_n^D$ of type $D_n$ consisting of all signed permutations $f=(\epsilon,\sigma)\in\SS_n^B$ with $\epsilon([n])=1$, where $\epsilon(T):=\prod_{j\in T}\epsilon(j)$ for all $T\subseteq[n]$.
We may assume $i\ge1$ and $(i,n)\ne \{(1,2), (2,4)\}$, thanks to Proposition~\ref{prop:AutV0Hn}, Corollary~\ref{cor:halved-cube}, and Example~\ref{example:halved-cube}.

\begin{theorem}
For $i=1,\ldots,\lfloor n/2 \rfloor$, every signed permutation $f=(\sigma,\epsilon)\in\SS_n^D$ induces an automorphism of the Norton algebra $V_i(\frac12 Q_n)$ by sending $\chi_S$ to $\epsilon(\sigma(S))\chi_{\sigma(S)}$ for all $i$-sets $S\subseteq[n]$.
Such automorphisms form a group isomorphic to $\SS_n^D$ if $i$ and $n$ are not both even and $(i,n)\ne(1,2)$, or $\SS_n^D/\{\pm1\}$ if $i$ and $n$ are both even and $(i,n)\ne(2,4)$.
\end{theorem}

\begin{proof}
Every signed permutation $f=(\sigma,\epsilon)\in\SS_n^D$ induces an automorphism of $V_i(\frac12 Q_n)$ since the following holds for all $i$-sets $S,T\subseteq[n]$.
\begin{itemize}
\item
If $|S\triangle T|=i$ then $f(\chi_S)\star f(\chi_T) = \epsilon\sigma(S)\epsilon\sigma(T)\chi_{\sigma(S)\triangle\sigma(T)} = \epsilon\sigma(S\triangle T)\chi_{\sigma(S\triangle T)} = f(\chi_S\star\chi_T)$ since $\epsilon\sigma(S)\epsilon\sigma(T)=\epsilon\sigma(S\triangle T)$.
\item
If $|S\triangle T|=n-i$ then $f(\chi_S)\star f(\chi_T) = \epsilon\sigma(S)\epsilon\sigma(T) \chi_{((\sigma(S)\triangle\sigma(T))^c} = \epsilon\sigma(S\triangle T)\chi_{\sigma((S\triangle T)^c)}$ and $f(\chi_S\star\chi_T) = \epsilon\sigma((S\triangle T)^c) \chi_{\sigma((S\triangle T)^c)}$ equal since $\epsilon([n])=1$.
\item
If $|S\triangle T|\notin\{i,n-i\}$ then $\chi_S\star\chi_T=0$ and $f(\chi_S)\star f(\chi_T) = \epsilon\sigma(S) \chi_{\sigma(S)} \star \epsilon\sigma(T) \chi_{\sigma(T)} = 0$.
\end{itemize}
For any $f'=(\sigma',\epsilon')\in\SS_n^D$ we have $ff' = (\epsilon\cdot \epsilon'\sigma^{-1}, \sigma\sigma')$ by Equation~\eqref{eq:SB} and thus
\[ f(f'(\chi_S)) = f(\epsilon'\sigma'(S)\chi_{\sigma'(S)}) =  \epsilon\sigma\sigma'(S) \epsilon'\sigma'(S) \chi_{\sigma\sigma'(S)} = (ff')(\chi_S). \]
Therefore we have a homomorphism from $\SS_n^D$ to the automorphism group of $V_i(\frac12 Q_n)$.
To find its kernel, suppose $f=(\sigma,\epsilon)\in\SS_n^D$ fixes $\chi_S$, i.e., $\sigma(S)\in\{S,S^c\}$ and $\epsilon\sigma(S)=1$ for any $i$-set $S\subseteq[n]$.

If $1\le i<n/2$ then we have $\sigma(S)=S$ since $|\sigma(S)|=i<n-i$.
If $3\le i=n/2$ then we also have $\sigma(S)=S$ since $\sigma(S)=S^c$ implies $\sigma(T) \notin\{T,T^c\}$ for the $i$-set $T\subseteq[n]$ obtained from $S$ by replacing any $j\in S$ with $k\in S^c-\{\sigma^{-1}(j)\}$.
Thus $\sigma=\id$ as $\SS_n$ acts faithfully on $i$-subsets of $[n]$.

For any distinct $j,k\in[n]$, there exists an $i$-set $S\subseteq[n]$ such that $j\in S$ and $k\notin S$.
Replacing $j$ with $k$ in $S$ gives an $i$-set $T\subseteq[n]$, and $\epsilon(S)=\epsilon(T)=1$ implies $\epsilon(j)=\epsilon(k)$.
It follows that $\epsilon=1$ or $\epsilon=-1$, and the latter implies that $n$ is even as $\epsilon([n])=1$.
Conversely, if $\epsilon=-1$ then $\epsilon(S)=1$ for all $i$-sets $S\subseteq[n]$ if and only if $i$ is even.
The result follows.
\end{proof}

\begin{example}
The signed permutation $f\in\SS_6^B$ given by $f(1)=-1$ and $f(j)=j$ for $j=2,3,4,5,6$ does not induce an automorphism of the Norton algebra $V_2(\frac12 Q_6)$, since $f$ fixes $\chi_{56}=\chi_{12}\star\chi_{34}$ but $f(\chi_{12})\star f(\chi_{34}) = - \chi_{12}\star \chi_{34} = - \chi_{56}$.
\end{example}

Next, we measure the nonassociativity of the Norton algebras of the halved cube $\frac12 Q_n$.

\begin{theorem}\label{thm:nonassoc-halved-cube}
For $i=0,1,\ldots,\lfloor n/2 \rfloor$, the Norton product on $V_i(\frac12 Q_n)$ is associative if $i=0$, $i$ is odd and $i< \lceil n/3 \rceil$, $i$ and $n-i$ are both odd, or $n-i$ is odd and $i>\lfloor 2n/3\rfloor$, or totally nonassociative otherwise.
\end{theorem}

\begin{proof}
If $i=0$ then $V_i(\frac12 Q_n)$ is one-dimensional and must be associative.
Assume $i\ge1$ below.
By Lemma~\ref{lem:Delta} and Theorem~\ref{thm:halved-cube}, there exist $i$-sets $S,T\subseteq [n]$ such that $\chi_S\star\chi_T\ne0$ if and only if either $i\le 2n/3$ and $i$ is even or $i\ge n/3$ and $n-i$ is even.
If this condition fails then $\star$ is zero and must be associative.
Otherwise there exist $i$-sets $S,T\subseteq [n]$ such that $\chi_S\star\chi_T=\chi_{S\triangle T} = \chi_{(S\triangle T)^c}$ belongs to the basis for $V_i(\frac12 Q_n)$ provided in Theorem~\ref{thm:halved-cube}.
The rest of the proof goes in the same way as the proof of Theorem~\ref{thm:NonAssocQn} for the hypercube $Q_n$.
\end{proof}

\subsection{Folded cube}

The \emph{folded cube} $\square_n$ can be obtained from the hypercube graph $Q_n$ by identifying each pair of vertices at distance $n$ from each other.
This is a distance regular graph of diameter $d=\lfloor n/2 \rfloor$ whose eigenvalues and multiplicities are~\cite[\S9.2]{DistReg1}
\[ \theta_i = n-4i \qand \dim(V_i) = \binom{n}{2i} \quad \text{for } i=0,1,\ldots,d=\lfloor n/2 \rfloor. \]

Equivalently, we can define the vertex set of the folded cube $\square_n$ to be $X:=\ZZ_2^{n-1}\times\{0\}$, and let two vertices be adjacent in $\square_n$ if they differ either in exactly one position or in all but the last position.
In other words, $\square_n$ is the Cayley graph $\Gamma(X, X_1\cup X_{n-1})$, where $X_i$ is the set of all elements in $X$ with exactly $i$ ones.
This allows us to determine the Norton algebra structure on each eigenspace of $\square_n$.
For every $S\subseteq[n]$ we define a linear character $\chi_S$ of $X$ by
\[ \chi_S(x) := \prod_{j\in S}(-1)^{x(j)} \quad\text{for all } x\in X. \]

\begin{theorem}\label{thm:folded-cube}
For $i=0,1,\ldots, d=\lfloor n/2 \rfloor$, the set $\{\chi_S: S\subseteq[n],\ |S|=2i\}$ is a basis for the eigenspace $V_i$ of the folded cube $\square_n$ such that for all $S,T\subseteq[n]$ with $|S|=|T|=2i$ we have 
\[ \chi_S \star \chi_T =
\begin{cases}
\chi_{S\triangle T} & \text{if } |S\triangle T|=2i \\
0 & \text{otherwise}.
\end{cases} \]
\end{theorem}

\begin{proof}
For any $S,T\subseteq[n]$, we have $\chi_S=\chi_T$ if and only if $S-\{n\}=T-\{n\}$.
Thus the set $\{\chi_S: S\subseteq[n]\}$ has cardinality $2^{n-1}=|X|$ and must equal the set $X^*$ of all linear characters of $X$, thanks to Theorem~\ref{thm:character}.
By Theorem~\ref{thm:Cayley}, this gives an eigenbasis of $\square_n$.
Any element of $X^*$ can be written as $\chi_S$ for some $S\subseteq[n]$ with an even cardinality $|S|=2i$ since adding or deleting $n$ from $S$ does not alter the corresponding character.
The eigenvalue associated with this eigenvector is
\[ \begin{aligned}
\chi_S(X_1\cup X_{n-1}) 
&= \sum_{j\in S-[n]} (-1) + \sum_{j\in[n-1]-S} 1 + \prod_{j\in S-\{n\}}(-1) \\
&= \begin{cases}
 -(2i-1) + n-2i - 1 = n-4i & \text{if } n\in S \\
 -2i + n-1-2i + 1 = n-4i & \text{if } n\notin S.
 \end{cases} 
 \end{aligned} \]
Thus $\{\chi_S: S\subseteq[n],\ |S| = 2i\}$ is a basis for the eigenspace $V_i$.
If $S$ and $T$ are $2i$-subsets of $[n]$ then $|S\triangle T| = |S|+|T|-2|S\cap T|=2j$ for some integer $j$ and $\chi_S\cdot \chi_T = \chi_{S\triangle T}\in V_j$.
Thus the projection of $\chi_S\cdot\chi_T$ to $V_i$ is either itself if $j=i$ or zero otherwise.
\end{proof}

\begin{remark}
Every element $u\in X=\ZZ_2^{n-1}\times\{0\}$ is uniquely determined by its support, and the linear character $\chi_u$ given in the proof of Theorem~\ref{thm:character} equals $\chi_S$, where $S:=\supp(u)$ if $|\supp(u)|$ is even or $S:=\supp(x)\cup\{n\}$ if $|\supp(u)|$ is odd.
\end{remark}

\begin{corollary}
For $i=0,1,\ldots,\lfloor n/2 \rfloor$, the Norton algebra $V_i(\square_n)$ is isomorphic to $V_{2i}(Q_n)$.
\end{corollary}

\begin{proof}
This follows immediately from Corollary~\ref{cor:hypercube} and Theorem~\ref{thm:folded-cube}.
\end{proof} 

\subsection{Folded half-cube}\label{sec:halved-folded-cube}

For any even integer $n\ge6$, the \emph{folded half-cube} $\frac12 \square_n$ has vertex set $X$ consisting of all even weighted elements of $\ZZ_2^{n-1}\times\{0\}$ and edge set $E$ consisting of all unordered pairs of vertices differing in exactly $2$ or $n-2$ positions.
This is a distance regular graph of diameter $d=\lfloor n/4 \rfloor$.
For $i=0,1,\ldots,d$, the $i$th eigenvalue and its multiplicity are~\cite[\S9.2D]{DistReg1}
\[ \theta_i= \frac{(n-4i)^2-n}2 
\qand \dim(V_i)=
\begin{cases}
\binom{n}{2i} & \text{if } i<n/4 \\
\frac12\binom{n}{2i} & \text{if } i=n/4.
\end{cases}\]

The folded half-cube $\frac12\square_n$ is the Cayley graph $\Gamma(X,X_2\cup X_{n-2})$ of the finite abelian group $X$ with respect to $X_2\cup X_{n-2}$, where $X_i$ is the set of all elements in $X$ with exactly $i$ ones.
For every $S\subseteq[n]$ we define a linear character $\chi_S$ of $X$ by 
\[ \chi_S(x) := \prod_{j\in S} (-1)^{x(j)} \quad\text{for all } x\in X. \]

\begin{theorem}\label{thm:halved-folded-cube}
The Norton algebra $V_i(\frac12\square_n)$ has a basis $\{\chi_S: S\subseteq[n],\ |S|=2i\}$ if $0\le i<n/4$ or $\{\chi_S: S\subseteq[n],\ |S|=2i,\ 1\in S\}$ if $i=n/4$.
For any $S,T\subseteq[n]$ with $|S|=|T|=2i$ we have
\[ \chi_S \star \chi_T =
\begin{cases}
\chi_{S\triangle T} = \chi_{(S\triangle T)^c} & \text{if } |S\triangle T|\in\left\{2i, n-2i \right\} \\
0 & \text{otherwise.} 
\end{cases} \]
\end{theorem}

\begin{proof}
One can check that $\chi_S=\chi_T$ if and only if $S-[n]=T-[n]$ or $S^c-[n]=T-[n]$.
Thus the set $\{\chi_S: S\subseteq[n]\}$ has cardinality $2^{n-2}=|X|$.
By Theorem~\ref{thm:character} and Theorem~\ref{thm:Cayley}, this set consists of all linear characters of $X$ and is an eigenbasis of $\frac12 \square_n$.
An element in this basis can be written as $\chi_S$ for some $2i$-set $S\subseteq[n]$ with $i\le n/4$, and it is an eigenvector of $\frac12\square_n$ corresponding to the eigenvalue
\[ \chi_S(X_2\cup X_{n-2}) = 
\begin{cases}
\binom{2i-1}{2} - (2i-1)(n-2i) + \binom{n-2i}{2} + 2i-1- (n-2i) = \frac{(n-4i)^2-n}{2} & \text{if } n\in S \\
\binom{2i}{2} - 2i(n-1-2i) + \binom{n-2i-1}{2} - 2i + (n-1-2i) = \frac{(n-4i)^2-n}{2} & \text{if } n\notin S.
\end{cases} \]
The desired basis for $V_i(\frac12\square_n)$ follows.

Let $S,T$ be $2i$-subsets of $[n]$.
Then $|S\triangle T|=|S|+|T|-2|S\cap T|=2j$ for some integer $j\ge0$. 
We have $\chi_S\cdot \chi_T = \chi_{S\triangle T} = \chi_{(S\triangle T)^c}$, which belongs to $V_j(\frac12\square_n)$ if $j\le \frac n4$ or $V_{\frac n2-j}(\frac12\square_n)$ if $f>\frac n4$. 
Thus the projection onto $V_i(\frac12\square_n)$ takes $\chi_S\cdot\chi_T$ to itself if $2j\in\{2i,n-2i\}$ or zero otherwise.
\end{proof}

\begin{corollary}
For $i=0,1,\ldots,\lfloor n/4\rfloor$, there is an algebra isomorphism $V_i(\frac12 \square_n) \cong V_{2i}(\frac12 Q_n)$.
\end{corollary}

\begin{proof}
This follows immediately from Theorem~\ref{thm:halved-cube} and Theorem~\ref{thm:halved-folded-cube}.
\end{proof}

\section{Bilinear forms graphs}\label{sec:Hq}

As a $q$-analogue of the Hamming graph $H(d,e)$, the \emph{bilinear forms graph} $H_q(d,e)$ has vertex set $X=\Mat_{d,e}(\FF_q)$ consisting of all $d\times e$ matrices over a finite field $\FF_q$ and edge set $E$ consisting of all unordered pairs of vertices whose difference is of rank one.
Assume $d\le e$, without loss of generality. 
The graph $H_q(d,e)$ is distance regular and has diameter $d$~\cite[\S9.5.A]{DistReg1} whose eigenvalues were computed by Delsarte~\cite{Delsarte} using a recursive relation.
One can also use a method for distance regular graphs with classical parameters~\cite[\S8.4]{DistReg1} to show that for $i=0,1,\ldots,d$, the $i$th eigenvalue of $H_q(d,e)$ and its multiplicity are
\[ \theta_i = \frac{q^{d+e-i}-q^d-q^e+1}{q-1}  \qand
\dim(V_i) = \qbinom{d}{i}_q(q^e-1)\cdots(q^e-q^{i-1}). \]
Here we use the \emph{$q$-binomial coefficient} 
\[ \qbinom{d}{i}_q:=\frac{(1-q^d)(1-q^{d-1})\cdots(1-q^{d-i+1})}{(1-q^i)(1-q^{i-1})\cdots(1-q)} \]
which counts $i$-dimensional subspaces in the vector space $\FF_q^d$.

Note that $\dim(V_i)$ is exactly the number of $d\times e$ matrices over $\FF_q$ with rank $i$ since the column space of such a matrix is a subspace of $\FF_q^d$ with dimension $i$, and after fixing any basis for this space, the columns are determined by an $i\times e$ matrix over $\FF_q$ of rank $i$. 
This is essentially the \emph{rank decomposition} of a $d\times e$ matrix $M$ of rank $i$ into a product $M=CF$ of a $d\times i$ matrix $C$ and a $i\times e$ matrix $F$;
this decomposition is unique up to a right multiplication of $C$ by some invertible $i\times i$ matrix $R$ and a left multiplication of $F$ by $R^{-1}$.

We study the eigenspaces of the bilinear forms graph $H_q(d,e)$ using the fact that $H_q(d,e)$ is the Cayley graph $\Gamma(X, X_1)$ of the finite abelian group $X=\Mat_{d,e}(\FF_q)\cong \FF_q^{de}$ with respect to $X_1$, where $X_i$ denotes the set of $d\times e$ matrices over $\FF_q$ with rank $i$.
Let $\omega:=\exp(2\pi i/q)\in\CC$ be a primitive $q$th root of unity.
For each $u\in X$, we have a linear character $\chi_u$ defined by 
\begin{equation}\label{eq:CharHq}
\chi_u(x) := \prod_{j=1}^d\prod_{k=1}^e \omega^{u_{jk} x_{jk}} = \omega^{ \tr(u^tx) } \quad\text{for all } x\in X. 
\end{equation}

\begin{theorem}
For $i=0,1,\ldots,d$, the (complex) eigenspace $V_i$ of the bilinear forms graph $H_q(d,e)$ has a basis $\{\chi_u: u\in X_i\}$ such that for all $u,v\in X_i$,
\[ \chi_u\star\chi_v = 
\begin{cases} 
\chi_{u+v} & \text{if } u+v \in X_i \\
0 & \text{otherwise}.
\end{cases} \]
\end{theorem}

\begin{proof}
By Theorem~\ref{thm:character} and Theorem~\ref{thm:Cayley}, we have an eigenbasis $\{\chi_u: u\in X\}$ with $\chi_u$ being an eigenvector of the bilinear forms graph $H_q(d,e)$ corresponding to the eigenvalue $\chi_u(X_1)$ for all $u\in X$.
To compute the eigenvalue, we apply the rank decomposition to each $x\in X_1$ and write it as $x=y^tz$ for some nonzero vectors $y=(y_1,\ldots,y_d)\in\FF_q^d$ and $z=(z_1,\ldots,z_e)\in \FF_q^e$ with the understanding that rescaling $y$ and $z$ to $cy$ and $c^{-1}z$ for any $c\in\FF_q^\times$ does not alter $x$. 
Then for any $u\in X_i$ we have
\[ \begin{aligned}
\chi_u(X_1) &= \sum_{x\in X_1} \prod_{j=1}^d \prod_{k=1}^e \omega^{u_{jk} x_{jk}} \\
&= \frac1{q-1} \left( \sum_{y\in\FF_q^d} \sum_{z\in \FF_q^e} \prod_{j=1}^d \prod_{k=1}^e \omega^{ u_{jk} y_jz_k} - q^d-q^e+1 \right)\\
&= \frac1{q-1} \left( \sum_{y\in\FF_q^d} \prod_{k=1}^e \sum_{z_k\in\FF_q} \omega^{\sum_{j=1}^d u_{jk} y_j z_k} - q^d-q^e+1 \right)\\
&= \frac1{q-1} \left( \left( \sum_{y\in\FF_q^d:\, yu=0} q^e \right) - q^d-q^e+1 \right)\\
&= \frac1{q-1} \left( q^{d+e-i} - q^d-q^e+1 \right)
\end{aligned} \]
where the second last equality holds by Equation~\eqref{eq:root}.
Hence $\{\chi_u: u\in X_i\}$ is a basis of $V_i$.
If $u,v\in X_i$ then $\chi_u \cdot \chi_v = \chi_{u+v}\in V_j$ where $j=\rank(u+v)$, and the projection onto $V_i$ either fixes $\chi_{u+v}$ if $i=j$ or annihilates it otherwise.
\end{proof}

\begin{remark}
The eigenvalue $\chi_u(X_1)$ is a special case of the computation by Delsarte~\cite[Theorem A.2]{Delsarte} using recursion, but our calculation in the above proof is more direct.
\end{remark}

Next, we study the automorphisms of the Norton algebra $V_i(H_q(d,e))$.
If $i=0$ then this algebra is spanned by an idempotent $\chi_0$ and thus has a trivial automorphism group.
For $i\ge1$ we have the following result, where $I_n$ denotes the $n\times n$ identity matrix.

\begin{theorem}\label{thm:AutHq}
For $i=1,\ldots,d$, the automorphism group of the Norton algebra $V_i(H_q(d,e))$ admits a subgroup isomorphic to $\Mat_{d,e}(\FF_q) \rtimes \left( (\GL_d(\FF_q) \times \GL_e(\FF_q))/\{(cI_d, cI_e): c\in\FF_q^\times\} \right)$.
\end{theorem}

\begin{proof}
First, every $x\in X=\Mat_{d,e}(\FF_q)$ induces an automorphism $\phi_x$ of $V_i(H_q(d,e))$ by sending $\chi_u$ to $\chi_x(u) \chi_u$ for all $u\in X_i$, as one can check the following for all $u,v\in X_i$.
\begin{itemize}
\item
If $u+v\in X_i$ then $\phi_x(\chi_u)\star\phi_x(\chi_v) = \chi_x(u) \chi_x(v) \chi_u\star\chi_v = \chi_x(u+v) \chi_{u+v} =\phi_x(\chi_u\star\chi_v)$.
\item
If $u+v\notin X_i$ then $\phi_x(\chi_u)\star\phi_x(\chi_v) = \chi_x(u) \chi_x(v) \chi_u\star\chi_v = 0 =\phi_x(\chi_u\star\chi_v)$.
\end{itemize}
For any $x,y\in X$, we have $\phi_x\phi_y=\phi_{x+y}$ since if $u\in X_i$ then
\[ \phi_x(\phi_{y}(\chi_u))=\phi_x(\chi_y(u)(\chi_u))=\chi_x(u)\chi_y(u)\chi_u = \chi_{x+y}(u) \chi_u = \phi_{x+y}(\chi_u). \]
Suppose that $\phi_x(\chi_u)=\chi_x(u)\chi_u = \chi_u$, i.e., $\chi_x(u)=1$ for all $u\in X_i$.
For any $j\in[d]$ and $k\in[e]$, we can construct $u\in X_i$ in such a way that changing its $(j,k)$-entry gives another matrix $u'$ with the same rank, and $\chi_x(u) =\chi_x(u')$ implies that the $(j,k)$-entry of $w$ is zero.
Therefore $x\mapsto \phi_x$ gives a monomorphism from $X$ to the automorphism group of $V_i(H_q(d,e))$.

Next, every $a\in\GL_d(\FF_q)$ induces an automorphism $\lambda_{a}$ of $V_i(H_q(d,e))$ by sending $\chi_u$ to $\chi_{au}$ for all $u\in X_i$ as the following holds for all $u,v\in X_i$.
\begin{itemize}
\item
If $u+v\in X_i$ then $\lambda_a(\chi_u) \star \lambda_a(\chi_v) = \chi_{au}\star\chi_{av} = \chi_{au+av} = \chi_{a(u+v)} = \lambda_{a}(\chi_u\star\chi_v)$.
\item
If $u+v\notin X_i$ then  $\lambda_{a}(\chi_u) \star \lambda_{a}(\chi_v) = \chi_{au}\star\chi_{av} = 0=\lambda_{a}(\chi_u\star\chi_v)$.
\end{itemize}
If $a,a'\in\GL_d(\FF_q)$ then $\lambda_a\circ\lambda_{a'}=\lambda_{aa'}$.
Suppose that $\lambda_a(\chi_u)=\chi_{au}=\chi_u$, i.e., $au=u$ for all $u\in X_i$.
For any $j\in[d]$, there exists $u\in X_i$ such that its $(1,k)$-entry is one if $k=j$ or zero if $k\ne j$, and $au=u$ implies that the $j$th column of $a$ coincides with the first column of $u$.
This shows that $a=I_d$.
Hence $a\mapsto\lambda_a$ gives a monomorphism from $\GL_e(\FF_q)$ to the automorphism group of $V_i(H_q(d,e))$.

Similarly, every $b\in \GL_d(\FF_q)$ induces an automorphism $\rho_b$ of $V_i(H_q(d,e))$ by sending $\chi_u$ to $\chi_{ub^{-1}}$ for all $u\in X_i$, and $b\mapsto\rho_b$ gives a monomorphism from $\GL_e(\FF_q)$ to the automorphism group of $V_i(H_q(d,e))$.
It is clear that $\lambda_a$ and $\rho_b$ commute.
Suppose that $\lambda_a=\rho_b$, i.e., $au=ub^{-1}$ for all $u\in X_i$.
Below we distinguish two cases to show that $a=cI_d$ and $b=cI_e$ for some $c\in\FF_q^\times$.
\begin{itemize} 
\item
Assume $i=d=e$.
Taking $u=I_d$ gives $a=b^{-1}$.
Then $au=ub^{-1}=ua$ for all $u\in X_i$ means that $a$ is in the center of $\GL_d(\FF_q)$, that is, $a=cI_d$ for some $c\in\FF_q^\times$.
\item
Assume $i<d$ or $d<e$.
For any $j,j'\in[d]$ and any $k,k'\in[e]$, let $a_{jj'}$ and $b^{-1}_{k,k'}$ denote the $(j,j')$-entry of $a$ and $(k,k')$-entry of $b^{-1}$.
If $j\ne j'$ and $k\ne k'$ then there exists $u\in X_i$ such that its $(j,k)$-entry is $1$ and all other entries on the $j$th and $j'$th rows and on the $k$th and $k'$th columns are zero, and taking the $(j,k)$-entry, $(j,k')$-entry, and $(j',k)$-entry of $au=ub^{-1}$ gives $a_{jj}=b^{-1}_{kk}$ and $a_{j'j}=b^{-1}_{kk'}=0$.
Thus $a=cI_d$ and $b^{-1}=cI_e$ for some $c\in\FF_q^\times$.
\end{itemize}

It follows that the group $G$ generated by $\{\lambda_a: a\in \GL_d(\FF_q)\}$ and $\{\rho_b:b\in \GL_e(\FF_q)\}$ is isomorphic to $(\GL_d(\FF_q)\times\GL_e(\FF_q) )/ \{ (cI_d, cI_e): c\in\FF_q^\times\}$.
This group has a trivial intersection with the group $H:=\{\phi_x: x\in X\}$, since if $\phi_x=\lambda_a\rho_b$ then $\chi_x(u) \chi_u = \chi_{aub^{-1}}$ for all $u\in X_i$ and $\phi_x$ must be the identity mapping.
Moreover, we have $\rho_b^{-1} \lambda_a^{-1} \phi_x \lambda_a\rho_b = \chi_{a^t x (b^{-1})^t}$ since if $u\in X_i$ then
\[ \rho_b^{-1} \lambda_a^{-1} \phi_x \lambda_a\rho_b (u) 
= \chi_x(aub^{-1}) \chi_{a^{-1}aub^{-1}b} = \chi_{a^t x (b^{-1})^t}(u) \chi_u \]
where the last equality holds by Equation~\eqref{eq:CharHq} and the fact that $\tr(x^t aub^{-1} ) = \tr(b^{-1} x^t a u)$.
Hence $HG$ contains $H$ as a normal subgroup and must be isomorphic to $H\rtimes G$.
\end{proof}

Now we measure the nonassociativity of the Norton product on $V_i(H_q(d,e))$.
The case $d=1$ is already done as $H_q(1,e)$ is a complete graph isomorphic to the Hamming graph $H(1,q^e)$.

\begin{theorem}
Assume $d\ge2$. 
Then the Norton product $\star$ on the eigenspace $V_i$ of $H_q(d,e)$ is associative if $i=0$ or totally nonassociative if $i=1,\ldots,d$.
\end{theorem}

\begin{proof}
The Norton product $\star$ on $V_0$ is associative as $V_0$ is one-dimensional.
Assume $i\ge1$ below.


For $q\ge4$, the proof of Theorem~\ref{thm:NonAssocH4} is still valid since  any $u\in X_i$ gives $cu\in X_i$ if $c\in[q-1]$.

For $q=3$, the proof of Theorem~\ref{thm:NonAssocH3} is still valid since one can find matrices $u, v, w\in X_i$ such that $u+v\notin X_i$, $u+w\in X_i$, and $v+w\in X_i$.

For $q=2$, the proof of Theorem~\ref{thm:NonAssocQn} is still valid since there exist matrices $u,v,w\in X_i$ such that $u+v=w$, which implies $u+w=v$ and $v+w=u$.

We leave the existence of $u,v,w$ in the above cases as an exercise to the reader.
\end{proof}

\section{Remarks and questions}\label{sec:questions}

In this paper we construct a basis for each eigenspace $V_i$ of the Hamming graph $H(n,e)$ using the linear characters of the vertex set $X=\ZZ_e^n$ treated as a finite abelian group.
Our basis is complex and can be converted to a real basis, but the existence of a basis over $\ZZ$ or even $\{0,1,-1\}$ needs further investigation.

We use our basis to provide a formula for the Norton product $\star$ on $V_i$ and obtain the following result on the automorphism group of the Norton algebra $(V_i,\star)$.
\begin{itemize}
\item
It is the trivial group if $i=0$.
\item
It is isomorphic to $\SS_e\wr\SS_n$ if $i=1$.
\item
It is isomorphic to $\SS_3\wr\SS_{2^{n-1}}$ if $e=3$ and $i=n$. 
\item
It is the general linear group of $V_i$ if $e=2$ and either $i>\lfloor 2n/3 \rfloor$ or $i$ is odd.
\item
It admits a subgroup is isomorphic to $(\ZZ_e\rtimes\ZZ_e^\times) \wr \SS_n$ if $e\ge3$ and $i\ge1$.
\item
It admits a subgroup is isomorphic to $\SS_n^B/\{\pm1\}$ if $e=2$ and $1\le i<n$ is even.
\end{itemize}
The groups mentioned above are all different, although most of them are related to the wreath product with symmetric groups. 
It will be interesting to see a complete description of the automorphism groups of all Norton algebras of the Hamming graph $H(n,e)$.

We also measure the nonassociativity of the Norton product on $V_i(H(n,e))$ and show that this commutative product as nonassociative as possible, except for some special cases in which it is either associative for trivial reasons (being zero or defined on a one-dimensional space) or equally as nonassociative as the double minus operation $a\ominus b:=-a-b$.
We are curious about whether the Norton algebras of other distance regular graphs are totally nonassociative in most cases and whether they could be related to the double minus operation or other elementary operations in some special cases.

Our results restrict to the hypercube $Q_n=H(n,2)$ and extend to the halved and/or folded cubes. 
We have algebra isomorphisms $V_i(\square_n)\cong V_{2i}(Q_n)$ and $V_i(\frac12 \square_n)\cong V_{2i}(\frac12 Q_n)$.
More generally, the folded graph $\overline\Gamma$ can be defined for any antipodal distance regular graph $\Gamma$ and is still distance regular~\cite[Proposition~2.14]{DistReg2}.
We would like to know if the Norton algebras of $\overline\Gamma$ and $\Gamma$ are related in the same way as above.

For other distance regular graphs, the linear character approach should be valid as long as they are also Cayley graphs of finite abelian groups, such as the square graph of the hypercube~\cite{SquareCube} and the alternating forms graphs~\cite[\S9.5B]{DistReg1}, 
but a different method may be necessary otherwise.
For instance, the Johnson graphs are not Cayley graphs in general, but they have been heavily studied and their spectra can be obtained by linear algebra (see Burcroff~\cite{Burcroff}) or representation theory (see Krebs and Shaheen~\cite{KrebsShaheen}).
It would be interesting to see whether the existing constructions of the eigenspaces of the Johnson graphs could be used to determine their Norton algebras.

\section*{Acknowledgements}

We thank Paul Terwilliger for encouraging conversations on this work. 
We also thank the developers of SageMath for providing such a useful computational tool.


\begin{thebibliography}{99}

\bibitem{Babai}
L. Babai, Spectra of Cayley graphs, J. Combin. Theory Ser. B {\bf 27} (1979), no.~2, 180--189. 

\bibitem{Subassociative}
M. S. Braitt\ and\ D. Silberger, Subassociative groupoids, Quasigroups Related Systems {\bf 14} (2006), no.~1, 11--26.

\bibitem{DistReg1}
A. E. Brouwer, A. M. Cohen\ and\ A. Neumaier, {\it Distance-regular graphs}, Ergebnisse der Mathematik und ihrer Grenzgebiete (3), 18, Springer-Verlag, Berlin, 1989. 

\bibitem{Spectra}
A. E. Brouwer\ and\ W. H. Haemers, {\it Spectra of graphs}, Universitext, Springer, New York, 2012. 

\bibitem{Burcroff}
A. Burcroff, Johnson schemes and certain matrices with integral eigenvalues, University of Michigan, Tech. Rep (2017).

\bibitem{Norton1}
P. J. Cameron, J.-M. Goethals\ and\ J. J. Seidel, The Krein condition, spherical designs, Norton algebras and permutation groups, Nederl. Akad. Wetensch. Indag. Math. {\bf 40} (1978), no.~2, 196--206. 

\bibitem{AssociativeSpectra1}
B. Cs\'{a}k\'{a}ny\ and\ T. Waldhauser, Associative spectra of binary operations, Mult.-Valued Log. {\bf 5} (2000), no.~3, 175--200.


\bibitem{DistReg2}
E.R. van Dam, J.H. Koolen, and H. Tanaka, Distance-regular graphs, Electron. J. Combin. \#DS22 (2016).

\bibitem{Delsarte}
Ph. Delsarte, Bilinear forms over a finite field, with applications to coding theory, J. Combin. Theory Ser. A {\bf 25} (1978), no.~3, 226--241.

\bibitem{AltForms}
P. Delsarte\ and\ J.-M. Goethals, Alternating bilinear forms over $GF(q)$, J. Combinatorial Theory Ser. A {\bf 19} (1975), 26--50.


\bibitem{CayleySum}
M. DeVos, L. Goddyn, B. Mohar, and R. \v S\'amal. Cayley sum graphs and eigenvalues of $(3,6)$-fullerenes, J. Combin. Theory Ser. B {\bf 99} (2009), no.~2, 358--369.

\bibitem{KrebsShaheen}
M. Krebs\ and\ A. Shaheen, On the spectra of Johnson graphs, Electron. J. Linear Algebra {\bf 17} (2008), 154--167. 

\bibitem{Lovasz79}
L. Lov\'{a}sz, {\it Combinatorial problems and exercises}, North-Holland Publishing Co., Amsterdam, 1979.

\bibitem{Lovasz75}
L. Lov\'{a}sz, Spectra of graphs with transitive groups, Period. Math. Hungar. {\bf 6} (1975), no.~2, 191--195.

\bibitem{CatMod}
N. Hein\ and\ J. Huang, Modular Catalan numbers, European J. Combin. {\bf 61} (2017), 197--218. 

\bibitem{VarCat}
N. Hein and J. Huang, Variations of the Catalan numbers from some nonassociative binary operations, arXiv:1807.04623.

\bibitem{Huang20}
J. Huang, Nonassociativity of the Norton Algebras of some distance regular graphs, Electron. J. Combin. \#P4.27 (2020).

\bibitem{DoubleMinus}
J. Huang, M. Mickey\ and\ J. Xu, The nonassociativity of the double minus operation, J. Integer Seq. {\bf 20} (2017), no.~10, Art. 17.10.3, 11 pp. 

\bibitem{DualPolarGraph}
F. Levstein, C. Maldonado\ and\ D. Penazzi, Lattices, frames and Norton algebras of dual polar graphs, in {\it New developments in Lie theory and its applications}, 1--16, Contemp. Math., 544, Amer. Math. Soc., Providence, RI. 

\bibitem{AssociativeSpectra2}
S. Liebscher\ and\ T. Waldhauser, On associative spectra of operations, Acta Sci. Math. (Szeged) {\bf 75} (2009), no.~3-4, 433--456.


\bibitem{NortonAlgebra}
C. Maldonado and D. Penazzi, Lattices and Norton algebras of Johnson, Grassmann and Hamming graphs, arXiv:1204.1947v1.

\bibitem{SquareCube}
S.M. Mirafzal, On the distance-transitivity of the square graphof the hypercube, arXiv: 2101.01615  

\bibitem{Sander}
T. Sander, Eigenspaces of Hamming graphs and unitary Cayley graphs, Ars Math. Contemp. {\bf 3} (2010), no.~1, 13--19.

\bibitem{OEIS}
N. J. A. Sloane, editor, The On-Line Encyclopedia of Integer Sequences, published electronically at \url{https://oeis.org}, 2018.

\bibitem{Norton2}
S. D. Smith, Nonassociative commutative algebras for triple covers of $3$-transposition groups, Michigan Math. J. {\bf 24} (1977), no.~3, 273--287.

\bibitem{A000975}
P. K. Stockmeyer, An exploration of sequence A000975, \textit{Fibonacci Quart.} \textbf{55} (2017), 174--185.

\bibitem{Terwilliger}
P. Terwilliger, The Norton algebra of a Q-polynomial distance-regular graph, arXiv:2006.04997.

\bibitem{AlmostAssociative}
T. Waldhauser, Almost associative operations generating a minimal clone, Discuss. Math. Gen. Algebra Appl. {\bf 26} (2006), no.~1, 45--73.

\bibitem{Supp}
A. Valyuzhenich\ and\ K. Vorob'ev, Minimum supports of functions on the Hamming graphs with spectral constraints, Discrete Math. {\bf 342} (2019), no.~5, 1351--1360.



\end{thebibliography}
\end{document}